\documentclass[a4paper,12pt,headsepline,cleardoubleempty,reqno,bibtotoc,tablecaptionabove]{amsart}
\usepackage[utf8]{inputenc}
\author{Roland Donninger}
\address{Universität Wien, Fakultät für Mathematik,
  Oskar-Morgenstern-Platz 1, 1090 Vienna, Austria}
\email{roland.donninger@univie.ac.at}
\author{David Wallauch}
\address{Universität Wien, Fakultät für Mathematik,
  Oskar-Morgenstern-Platz 1, 1090 Vienna, Austria}
\email{david.wallauch@univie.ac.at}
\thanks{Both authors are supported by the Austrian Science Fund FWF, Project P 30076: “Self-similar blowup in dispersive
	wave equations”.}
\title{Blowup behavior of strongly perturbed wave equations}
\usepackage{amsmath,amsfonts,amssymb,amsthm,empheq}
\usepackage{xcolor}
\usepackage{hyperref}
\usepackage{array}
\usepackage{fullpage}
\numberwithin{equation}{section}
\usepackage{graphicx}

\newcommand{\C}{\mathbb{C}}

\newcommand{\fpc}{\cos (t |\nabla | )}
\newcommand{\fps}{\frac{\sin(t| \nabla | )}{|\nabla |}}
\newcommand{\fpsi}{\frac{\sin((t-s)| \nabla | )}{|\nabla |}}
\newcommand{\R}{\mathbb{R}}

\newtheorem{thm}{Theorem}[section]
\newtheorem{defi}{Definition}[section]
\newtheorem{asm}{Assumption}[section]

\newtheorem{prop}{Proposition}[section]
\newtheorem{lem}{Lemma}[section]
\newcommand{\N}{\textup{\textbf{N}}}
\newcommand{\rg}{\textup{\textbf{rg}}}
\newcommand{\Span}{\textup{\textbf{span}}}

\newcommand{\Po}{\textup{\textbf{P}}_0}
\newcommand{\Pro}{\textup{\textbf{P}}}
\newcommand{\Qo}{\textup{\textbf{Q}}_0}
\newcommand{\Cf}{\textup{\textbf{C}}}
\newcommand{\Qr}{\textup{\textbf{Q}}}

\newcommand{\X}{\mathcal{X}}
\newcommand{\Mf}{\textup{\textbf{M}}}
\newcommand{\Nf}{\textup{\textbf{N}}}
\newcommand{\Lf}{\textup{\textbf{L}}}
\newcommand{\Vf}{\textup{\textbf{V}}}

\newcommand{\Jf}{\textup{\textbf{J}}}
\newcommand{\G}{\textup{\textbf{G}}}

\newcommand{\K}{\textup{\textbf{K}}}
\newcommand{\Sf}{\textup{\textbf{S}}}
\newcommand{\I}{\textup{\textbf{I}}}
\newcommand{\Uf}{\textup{\textbf{U}}}
\newcommand{\uf}{\textup{\textbf{u}}}
\newcommand{\vf}{\textup{\textbf{v}}}
\newcommand{\Rf}{\textup{\textbf{R}}}
\newcommand{\Wf}{\textup{\textbf{W}}}
\newcommand{\gf}{\textup{\textbf{g}}}
\newcommand{\rf}{\textup{\textbf{r}}}
\newcommand{\ff}{\textup{\textbf{f}}}
\newcommand{\Pf}{\textup{\textbf{P}}}
\newcommand{\Qf}{\textup{\textbf{Q}}}
\renewcommand{\Re}{\textup{Re}}
\renewcommand{\Im}{\textup{Im}}
\newcommand{\B}{\mathbb{B}}

\begin{document}
\maketitle
\begin{abstract}
We study the blowup behavior of a class of strongly perturbed wave
equations with a focusing supercritical power nonlinearity in three
spatial dimensions. We show that the ODE blowup profile of the
unperturbed equation still describes the asymptotics of stable
blowup. As a consequence, stable ODE-type blowup is seen to be a universal
phenomenon that exists in a large class
of semilinear wave equations. 
\end{abstract}

\section{Introduction}
Nonlinear wave equations describe a wide variety of phenomena in
fields ranging from
fundamental physics to the applied sciences, e.g.~general
relativity, quantum field theory, solid state physics, and nonlinear
optics. Typically, the equations that occur in applications
are way too complicated for a rigorous mathematical
analysis. One therefore resorts to toy models that are supposed to
capture and isolate essential features of the more complicated equations. From the
point of view of applications this strategy is only meaningful if the
phenomena discovered in the toy model are stable under perturbations
of the equation.

In the present paper we focus on the formation of singularities
(or blowup) in finite
time. The basic semilinear wave equation
\[ \Box u(t,x):=(\partial_t^2-\Delta_x)u(t,x)=u(t,x)|u(t,x)|^{p-1}, \quad p>1 \]
admits the explicit \emph{ODE blowup}
given by
\[ u(t,x)=\left (\frac{2(p+1)}{(p-1)^2}\right
)^\frac{1}{p-1}(1-t)^{-\frac{2}{p-1}}, \]
which is known to be stable under perturbations of the initial data
\cite{MerZaa03, MerZaa05, DonSch16}. It is thus natural to ask
whether this type of blowup is also relevant for more complicated
equations that occur in applications. In this paper we show that the
stable ODE blowup persists if one perturbs the equation in a very general
way.
Roughly speaking, we
consider equations of the form
\[ \Box u+F(u,\partial u)=u|u|^{p-1} \]
where $F$ is (at most) linear in the derivatives $\partial u$ and
satisfies some mild, natural requirements. We do not impose any smallness
assumption on $F$. Our result covers the whole range
$p>3$ in $3$ space dimensions and we allow for complex-valued
solutions.
A random example of an equation that we can cover would be
\[ \Box u(t,x)+t^5e^{it+|x|^2} u(t,x)\partial_t
  u(t,x)+u(t,x)^6=u(t,x)|u(t,x)|^6. \]
For the sake of simplicity we restrict ourselves to
radial solutions but the extension to the general case is purely technical.

The mechanism we exploit is most easily explained by considering the
Klein-Gordon equation
\[ \Box u+mu=u|u|^{p-1}. \]
The natural scaling transform related to the pure wave case
($m=0$) is given by
\[ u(t,x)\mapsto u_\lambda(t,x):=\lambda^{-\frac{2}{p-1}}u\left
(\tfrac{t}{\lambda}, \tfrac{x}{\lambda}\right ). \]
Under this scaling, the Klein-Gordon equation transforms as
\[ \Box u_\lambda+\lambda^{\frac{2p}{p-1}}mu_\lambda=u_\lambda|u_\lambda|^{p-1} \]
and if $\lambda\to 0$, the mass term becomes negligible.
The ODE blowup is self-similar and effectively, the
solution moves to
smaller and smaller scales as the blowup time is approached. 
The heuristic scaling analysis therefore suggests that the mass term
(and much more general perturbations) can be neglected
close to the blowup time.
We implement this idea rigorously by a purely perturbative argument.
Consequently, we do not make use of fragile structural
properties like Lyapunov functionals or virial identities. That is why we are able to treat
very general perturbations and all $p>3$. Our result shows, for the
first time in the \emph{supercritical context},
that stable ODE blowup is a \emph{universal phenomenon} that occurs in a large
class of models. 
\subsection{Setup}
Since we restrict ourselves to the radial case, the effective Cauchy
problem we study is given by
\begin{equation}\label{startingeq}
\begin{cases}
&\left(\partial_t^2 -\partial_r^2-\frac{2}{r}\partial_r \right)u(t,r)+ F(t,r,u(t,r),\partial_t u(t,r),\partial_r u(t,r)) =|u(t,r)|^{p-1}u(t,r)\\
&u(T_0,r)=f(r)\\ 
&\partial_0 u(T_0,r)=g(r),
\end{cases}
\end{equation}
where $T_0$ is some initial time which we will specify
below. Furthermore, $p$ is a constant that satisfies $3<p$. We
additionally assume that $f$ and $g$ are complex-valued initial data
and that $F$ satisfies some natural constraints.
Note that in the unperturbed case $(F=0)$, Eq.~\eqref{startingeq} has a conserved energy given by 
\begin{equation*}
\mathcal{E}(u(t,.),\partial_t u(t,.))=\frac{1}{2}\|(u(t,.),\partial_t u(t,.)\|^2_{(\dot{H}^1\times L^2)(\R^3)}-\frac{1}{p-1}\|u(t,.)\|^{p+1}_{L^{p+1}(\R^3)}.
\end{equation*}
Under the transformation
$$
u(t,r)\mapsto u_\lambda(t,r):=\lambda^{-\frac{2}{p-1}}u\left(\frac{t}{\lambda},\frac{r}{\lambda}\right)
$$
this energy scales as 
$$
\mathcal{E}(u_\lambda(t,.),\partial_t u_\lambda(t,.))=\lambda^{\frac{p-5}{p-1}}\mathcal{E}(u(t,.),\partial_t u(t,.)),
$$
while Eq.~\eqref{startingeq}
remains invariant if it is unperturbed.
Therefore, we say that Eq.~\eqref{startingeq} is subcritical for $1<p<5$, critical for $p=5$, and supercritical for $p>5$.\\
We also remark that by employing the wave propagators $\fpc$ and $\fps$, a weak formulation of Eq.~\eqref{startingeq} is given by
\begin{align*}
u(t,r)=&\fpc f+ \fps g +\int_0^t \fpsi\mathcal{N}(u)(s,r) ds,
\end{align*}
where 
$$\mathcal{N}(u)(s,r) =(|u(s,r)|^{p-1}u(s,r)-F(s,r,u(s,r),\partial_s u(s,r),\partial_r u(s,r))).$$
This weak formulation now has the advantage that instead of having to deal with the differential equation itself, one obtains a fixed point problem. In order to be able to find a fixed point, it is vital to work in a space with enough regularity to control the (possibly) supercritical nonlinearity. In our case $H^2 \times H^1$ will be sufficient. Recall that for $F=0$, Eq.~\eqref{startingeq} has an explicit blow up solution given by
\begin{equation}
u^T(t,r) =\kappa_p (T-t)^{-\frac{2}{p-1}}
\end{equation}
where
$$
\kappa_p=\left(\frac{2(p+1)}{(p-1)^2}\right)^{\frac{1}{p-1}}.
$$
For notational convenience we further set $c_p=\kappa_p^{p-1}$.
Note that since we allow for complex-valued solutions, the phase shift $u\mapsto e^{i\theta} u$, $\theta\in \R$, is another symmetry that leaves the unperturbed equation invariant.
This leads to a two parameter family of blowup solutions given by$$
u^T_{\theta}:= e^{i\theta}u^T.$$
Our interest in complex-valued solutions stems from the fact that, as
a special case of Eq.~\eqref{startingeq}, we obtain a semilinear
Klein-Gordon equation.

Now we turn or attention to the perturbation term
$F(t,r,u(t,r),\partial_t u(t,r),\partial_r u(t,r))$.
 First of all, we assume $F$ to be of the form
 \begin{equation}
   \label{eq:formF}
F(t,r,u,v,w)= A(t,r,u)+B(t,r,u)v+C(t,r,u)w,
\end{equation}
where $B$ and $C$ satisfy
\begin{equation}
  \label{eq:estBC}
  \begin{split}
|B(t,r,u)|+|C(t,r,u)|&\leq M(1+ |u|)
\\
|B(t,r,u_1)-B(t,r,u_2)|+|C(t,r,u_1)-C(t,r,u_2)|&\leq M |u_1-u_2|,
\end{split}
\end{equation}
for some $M>0$.
Next, $F$ needs to grow slowlier in $u$ than the leading nonlinearity itself. Concretely, there needs to be a constant $1\leq q<p$ such that $A$ satisfies
\begin{equation}
  \begin{split}
\label{estimateV1}
|A(t,r,u)|&\leq M(1+|u|^{q})
\\
|A(t,r,u_1)-A(t,r,u_2)|
&\leq M\left|u_1|u_1|^{q-1}-u_2|u_2|^{q-1}\right|.
\end{split}
\end{equation}
Since we will have to control $F$ in $H^1$, these constraints alone do
not suffice and we also have to impose restrictions on the
derivatives. As $u$ is a complex variable, we decompose it according
to $u=x+ iy$ and require $F$ to satisfy the bounds
\begin{equation}
  \label{eq:estderF}
  \begin{split}
|\partial_r F(t,r,u,v,w)|\leq& M(1+|u|^{q}+|v|+|w|)\\
|\partial_{x} F(t,r,x+ iy,v,w)|+|\partial_{y} F(t,r,x + iy,v,w)|\leq&
M(1+|u|^{q-1}+|v|+|w|).
\end{split}
\end{equation}
Finally, we will also need the Lipschitz-type estimates 
\begin{equation}
  \label{eq:estLipF}
  \begin{split}
&|\partial_rF(t,r,u_1,v_1,w_1)-\partial_r F(t,r,u_2,v_2,w_2)|
\\
&\leq M \big( \left|u_1|u_1|^{q-1}-u_2|u_2|^{q-1} \right|
+ |v_1-v_2|+|w_1-w_2|\\
&\;\;\;+|u_1v_1-u_2v_2|+|u_1w_1-u_2w_2|\big)
\\
&|\partial_{x_1}F(t,r,x_1+i y_1,v_1,w_1)- \partial_{x_2}F(t,r,x_2+iy_2,v_2,w_2)|
\\
&\leq M\left(\left| u_1|u_1|^{q-2}-u_2|u_2|^{q-2}\right|+|v_1-v_2|+|w_1-w_2|\right)
\\
&|\partial_{y_1}F(t,r,x_1+i y_1,v_1,w_1)- \partial_{y_2}F(t,r,x_2+iy_2,v_2,w_2)|
\\
&\leq M\left(\left|u_1|u_1|^{q-2}-u_2|u_2|^{q-2}\right|+|v_1-v_2|+|w_1-w_2|\right).
\end{split}
\end{equation}

\begin{asm}
  \label{asm:F}
  There exist constants $t_0, r_0, M>0$ and $q\in [1,p)$ such that
  \[ F: [1-t_0,1+t_0]\times [0,r_0]\times \C\times \C\times \C\to \C \] satisfies
  \eqref{eq:formF},
  \eqref{eq:estBC}, \eqref{estimateV1}, \eqref{eq:estderF}, and
  \eqref{eq:estLipF} for $t\in [1-t_0,1+t_0]$, $r\in [0,r_0]$ and
  \begin{align*}
    u, u_1,
    u_2, v, v_1, v_2, w, w_1, w_2&\in \C \\
    x,x_1,x_2,y,y_1,y_2&\in\R.
                         \end{align*}
\end{asm}
Due to finite speed of propagation it makes sense to study the Cauchy problem for Eq.~\eqref{startingeq} in the backwards lightcone 
$$
\Gamma^{T}_{T_0}:=\{(t,r):t\in(T_0,T),\,r \in [0,T-t]\},
$$
to which we restrict ourselves.
Our precise notion of solutions in the lightcone will be introduced later in Subsection \ref{sgt}. Nevertheless, we can already state the main theorem of this work.
\begin{thm}\label{maintheorem}
	Let $p >3$ and suppose that Assumption \ref{asm:F} holds. Then there exist constants $\delta,
        C,\varepsilon, \omega>0$ and $c\geq 1$ such that if $T_0 \in [1-\frac{3\delta}{c},1-\frac{2\delta}{c}]$ the following holds. Let $(f,g)$ be  initial data that satisfy 
	$$
	\|(f,g)-(u_0^1(T_0,.),\partial_0 u_{0}^1(T_0,.))\|_{H^2\times H^1(\B^3_{R})}<\varepsilon,
	$$
	with $R=\frac{3\delta}{c}+\frac{\delta}{c^2}$.
	Then there exist a $T\in
        [1-\frac{\delta}{c^2},1+\frac{\delta}{c^2}]$, a $C_0>0$, and a $\theta \in (-Cp\delta,Cp\delta)$, such that Eq.~\eqref{startingeq} has a unique solution $u: \Gamma^T_{T_0} \rightarrow \C$ that satisfies
	\begin{align*}
	(T-t)^{\frac{1}{2}+\frac{2}{p-1}}\|\left(u(t,.),\partial_t u(t,.)\right)\|_{\dot{H}^2\times \dot{H}^1(\B^3_{T-t})}&\leq C_0(T-t)^{\omega}\\
		(T-t)^{-\frac{1}{2}+\frac{2}{p-1}}\|(u(t,.),\partial_t u(t,.))-\left(u^T_\theta(t,.),\partial_t u^T_{\theta}(t,.)\right)\|_{\dot{H}^1\times L^2(\B^3_{T-t})}&\leq C_0(T-t)^{\omega}\\
		(T-t)^{-\frac{3}{2}+\frac{2}{p-1}}\|u(t,.)-u^T_\theta(t,.)\|_{L^2(\B^3_{T-t})}&\leq C_0(T-t)^{\omega},
	\end{align*}
	for all $t \in [T_0,T)$.
      \end{thm}
In particular, Theorem \ref{maintheorem} shows that the solution blows
up as $t\to T-$ with $u^T_\theta$ as an asymptotic profile.
Consequently, while $u^T_\theta$ does not actually solve Eq.~\eqref{startingeq}, it still provides the asymptotic blowup profile for initial data close to $u^1_0[0]$.
Furthermore, the blowup function $u^T_\theta$ satisfies
\begin{equation}
\|u^T_\theta(t,.)\|_{L^2(\B^3_{T-t})}\simeq (T-t)^{\frac{3}{2}-\frac{2}{p-1}},
\end{equation}
which makes the normalization factors appear naturally. We also remark
that, as we exclusively work with radial functions, i.e., $f(x)= \tilde{f}(|x|)$, we will throughout this paper identify $f$ with $\tilde{f}$.  
Note that for any radial function $f \in H^2(\B^3_R) $ we have that
$$
\|f\|_{H^2(\B^3_R)}^2\simeq \int_0^R r^2 (|f(r)|^2+|f'(r)|^2+|f''(r)|^2) d r,
$$
for any $R>0$.
\subsection{Related results}
The study of blowup solutions for semilinear wave equations has
attracted a lot of interest in recent years and due to the sheer
volume of results, we only mention a handful of works that deal with
ODE-type blowup. In the
unperturbed case, many results concerning the stability of the ODE blowup
are available. The subcritical case was thoroughly studied by Merle
and Zaag \cite{MerZaa03,MerZaa05,MerZaa18}, see also the work by
Alexakis and Shao \cite{SpySha17} and Azaiez \cite{Aza15}.
Furthermore, in the one-dimensional
case, Merle and Zaag were able to give a fairly complete picture of the blowup
behavior \cite{MerZaa07,MerZaa08,MerZaa12a,MerZaa12b}. They also
managed to extend some of these results to higher dimensions
\cite{MerZaa15,MerZaa16}. In the supercritical case, a very
influential numerical paper is \cite{BizChmTab04} by Bizo\'n, Chmaj,
and Tabor. Rigorous results were
established in \cite{DonSch14} and \cite{DonSch16} by the first author
and Schörkhuber. Recently, the stability for the critical equation in
three and five dimensions was shown in optimal regularity by proving
Strichartz estimates \cite{Don17,DonRao2020}. In the subcritical case,
Hamza
et.~al.~\cite{HamSai15,Ham16,HamSai14,HamZaa12a,HamZaa12b,HamZaa20}
studied the blowup behavior under various perturbations of the
equation, see also the paper by Killip, Stovall, and Visan
\cite{KilStoVis14} on blowup
bounds for the Klein-Gordon equation. Very recently, Speck
\cite{Speck20} studied ODE-type blowup in a class of quasilinear wave equations.

\subsection{Preliminary transformations}
Before we start analyzing Eq.~\eqref{startingeq} a few preliminary transformations are in order. We begin with transforming to the similarity coordinates, which are given by 
$$\tau=-\log(T-t)+\log (T-T_0),\,\,\,\, \rho=\frac{r}{T-t},$$
and setting $$\psi(\tau,\rho):=(T-T_0)^{\frac{2}{p-1}}e^{-\frac{2}{p-1}\tau}u(T-(T-T_0)e^{-\tau},(T-T_0)e^{-\tau}\rho)$$ as well as 
\begin{align*}
W(\psi)(\tau,\rho):=&(T-T_0)^{2+\frac{2}{p-1}}e^{-(2+\frac{2}{p-1}) \tau}F\bigg(T-(T-T_0)e^{-\tau},(T-T_0) \rho e^{-\tau},\\
&(T-T_0)^{-\frac{2}{p-1}}e^{\frac{2}{p-1}\tau}\psi(\tau,\rho),(T-T_0)^{-(1+\frac{2}{p-1})}e^{(1+\frac{2}{p-1})\tau}(\partial_\tau+\rho \partial_\rho)\psi(\tau,\rho),\\&(T-T_0)^{-(1+\frac{2}{p-1})}e^{(1+\frac{2}{p-1})\tau}\partial_\rho \psi(\tau,\rho)\bigg).
\end{align*}
In these coordinates Eq.~\eqref{startingeq} reads
\begin{align}
\left(\partial_\tau ^2+\frac{p+3}{p+1} \partial_{\tau}+2 \rho \partial_\rho \partial_\tau -(1-\rho^2)\partial_\rho^2+\frac{2(p+1)}{p-1}\rho\partial_\rho - \frac{2}{\rho} \partial_\rho\right)\psi(\tau,\rho)\nonumber \\-W(\psi)(\tau,\rho)=\psi(\tau,\rho)|\psi(\tau,\rho)|^{p-1}.
\end{align}
We further set \begin{align*}
&\psi_1(\tau,\rho):=\psi(\tau,\rho)\\
&\psi_2(\tau,\rho):=(\partial_\tau +\rho \partial_\rho+ \frac{2}{p-1})\psi(\tau,\rho),
\end{align*}
to obtain a first-order system given by
\begin{align}\label{equation3}
\begin{cases}
&\partial_\tau \psi_1 = \psi_2- \rho \partial_\rho \psi_1-\frac{2}{p-1}\psi_1\\
&\partial_\tau \psi_2= \partial_\rho^2 \psi_1+\frac{2}{\rho}\partial_\rho\psi_1-\rho \partial_\rho \psi_2- \frac{p+1}{p-1}\psi_2-W(\psi_1)- \psi_1|\psi_1|^{p-1}
\end{cases}
\end{align}
Note that in these coordinates the blowup $u^T_\theta$ corresponds to
$$(\psi_{\theta_1},\psi_{\theta_2})=e^{i\theta}(\kappa_p,\frac{2}{p-1}\kappa_p).
$$
Since solutions can take complex values, we will now split up $\psi_j$ into its real and imaginary part, respectively. This will be needed later on, when we linearize the nonlinearity, which is not holomorphic and  hence has to be linearized as a mapping from $\R^2$ to $\R^2$. Denote by $\varphi_j$ the real part of $\psi_j$ and by $\nu_j$ the imaginary part of $\psi_j$. Then Eq.~\eqref{equation3}, together with the initial data, reads
\begin{equation}\label{eq2}
\begin{cases}
&\partial_\tau \varphi_1 = \varphi_2- \rho \partial_\rho \varphi_1-\frac{2}{p-1}\varphi_1\\
&\partial_\tau \varphi_2= \partial_\rho^2 \varphi_1+\frac{2}{\rho}\partial_\rho \varphi_1-\rho \partial_\rho \varphi_2- \frac{p+1}{p-1}\varphi_2-\Re(W(\varphi_1+i\nu_1)- \varphi_1|\varphi_1+i\nu_1|^{p-1})\\
&\partial_\tau \nu_1 = \nu_2- \rho \partial_\rho \nu_1-\frac{2}{p-1}\nu_1\\
&\partial_\tau \nu_2= \partial_\rho^2 \nu_1+\frac{2}{\rho}\partial_\rho \nu_1-\rho \partial_\rho \nu_2- \frac{p+1}{p-1}\nu_2-\Im\left(W(\varphi_1+i\nu_1)- i\nu_1|\varphi_1+i\nu_1|^{p-1}\right)\\
&\varphi_1(0,\rho)=\Re\left((T-T_0)^{\frac{2}{p-1}}f((T-T_0)\rho)\right)\\
&\varphi_2(0,\rho)=\Re\left((T-T_0)^{\frac{p+1}{p-1}}g((T-T_0) \rho)\right)\\
&\nu_1(0,\rho)=\Im\left((T-T_0)^{\frac{2}{p-1}}f((T-T_0) \rho)\right)\\
&\nu_2(0,\rho)= \Im\left((T-T_0)^{\frac{p+1}{p-1}}g((T-T_0)\rho)\right)
.\end{cases}
\end{equation}
Note that, since we are only interested in values of $T$ that are close enough to 1, we can assume that $T \in [1-\frac{\delta}{c^2},1+\frac{\delta}{c^2}]$ and $T_0 \in [1-\frac{3\delta}{c},1-\frac{2\delta}{c}]$ for some $\delta$ and $c$ that will be specified later.
As we intend to study solutions that are close to the family of blowup functions, we will later on also make use of the splitting 

\begin{equation}\label{splitting}
\begin{cases}
&\varphi_1(0,\rho)=\Re\left((T-T_0)^{\frac{2}{p-1}}\tilde{f}((T-T_0) \rho)\right)-\kappa_p \left(\frac{T-T_0}{1-T_0}\right)^{\frac{2}{p-1}}\\
&\varphi_2(0,\rho)=\Re\left((T-T_0)^{\frac{p+1}{p-1}}\tilde{g}((T-T_0)\rho)\right)-\frac{2 \kappa_p}{p-1}\left(\frac{T-T_0}{1-T_0}\right)^{\frac{p+1}{p-1}}\\
&\nu_1(0,\rho)=\Im\left((T-T_0)^{\frac{2}{p-1}}\tilde{f}((T-T_0)\rho)\right)\\
&\nu_2(0,\rho)= \Im\left((T-T_0)^{\frac{p+1}{p-1}}\tilde{g}((T-T_0)\rho)\right)
\end{cases}
\end{equation}

%due to which the initial data to equation /\ref{eq2}) can be equivalently written down as
%\begin{align*}
%\begin{cases}
%&\phi_1 (0,x) =\Re\left(x T^{\frac{p+1}{p-1}}g( Tx)\right)\\
%&\phi_2(0,x)=\Re\left( T^{\frac{2}{p-1}}[Tx f'(Tx)+f(Tx)]\right)\\
%&\nu_1(0,x) =\Im\left(x T^{\frac{p+1}{p-1}}g( Tx)\right)\\
%&\nu_2(0,x)=\Im\left( T^{\frac{2}{p-1}}[Tx f'(Tx)+f(Tx)]\right).\end{cases}
%\end{align*}
\section{Linear theory}
With these preliminaries out of the way we will now analyze the linear part of Eq.~\eqref{eq2}.
To do so, we define the space $\mathcal{H}$ as
$$ \mathcal{H}:=\{\uf \in \left(H^2\times H^1(\B_1^3)\right)^2: \uf\text{  radial}\}
$$
together with the standard inner product, which we denote by $(.|.)$. Accordingly, we denote the corresponding norm by $\|.\|$. 
\subsection{Semigroup theory}
This setup now enables us to show  that the unbounded operator corresponding to the linear part in (\ref{eq2}), equipped with a proper domain, is closable and that its closure generates a $C_0$-semigroup. To that end we define the operator $\tilde{\Lf}$ with $$
D(\tilde{\Lf}):= \{\uf \in (C^3\times C^2(\overline{\B_1^3}))^2: \uf \text{  radial}\} $$
by setting
\begin{align*}
\tilde{\Lf} \uf (\rho)=\begin{pmatrix}
-\rho u_1'(\rho)+u_2(\rho)-\frac{2}{p-1}u_1(\rho)\\
u_1''(\rho)+\frac{2}{\rho}u_1'(\rho)-\rho u_2'(\rho)-\frac{p+1}{p-1}u_2(\rho)\\
-\rho u_3'(\rho)+u_4(\rho)-\frac{2}{p-1}u_3(\rho)\\
u_3''(\rho)+\frac{2}{\rho}u_3'(\rho)-\rho u_4'(\rho)-\frac{p+1}{p-1}u_4(\rho)\\
\end{pmatrix},
\end{align*}
for any $\uf \in D(\tilde{\Lf})$. 
This operator maps  $D(\tilde{\Lf})$ into $\mathcal{H}$ such that the following holds.
\begin{lem}\label{Semigroupgen}
	The operator $\tilde{\Lf}:D(\tilde{\Lf})\to \mathcal{H}$ is closable and 
	its closure $ \Lf$ generates a $C_0$-semigroup $\Sf:[0,\infty) \to \mathcal{B}\left(\mathcal{H}\right)$ with
	\begin{equation}
	\|\Sf(\tau)\|\lesssim e^{-\frac{2}{p-1}\tau}.
	\end{equation}
	Furthermore any $\uf \in D(\Lf)$ satisfies
	$\uf \in C(\overline{\B_1^3})^4\cap C^1(\B^3_1)^4$ and $u_1(0)=u_3(0)=u_2'(0)=u_4'(0)=0$.  
\end{lem}
\begin{proof}
The statement follows by combining the arguments from the proofs of Proposition 2.1 in \cite{Don17} and Lemma 2.2 in \cite{DonSch14}.
\end{proof}

This result also implies a useful bound on the resolvent of $\Lf$, which we will need later on.

\begin{lem}\label{resbound}
	The resolvent operator of $\Lf$, denoted by $\Rf_{\Lf}(\lambda) $, satisfies 
	$$
	\|\Rf_{\Lf}(\lambda)\| \lesssim \dfrac{1}{\Re(\lambda)+\frac{2}{p-1}},
	$$
	for any $\lambda \in \C$ with $ \Re(\lambda)>-\frac{2}{p-1}$.
\end{lem}
\begin{proof}
	This is immediate from the previous result and standard semigroup theory (see for instance \cite{EngNag99} p.55,
	Theorem 1.10).
\end{proof}

\subsection{The modulation ansatz}  \label{sgt}
In addition to $\Lf$ we define the operator $\Nf$ by
$$
\Nf(\uf)(\rho):=\begin{pmatrix}
0\\
\left|\left( u_1(\rho)\, ,u_3(\rho)\right)\right|^{p-1}u_1(\rho)\\
0\\
\left|\left( u_1(\rho)\, ,u_3(\rho)\right)\right|^{p-1}u_3(\rho)\\
\end{pmatrix},
$$
with $D(\Nf)=\mathcal{H}$. That we can indeed define $\Nf$ on the whole space $\mathcal{H}$ will follow from Lemma \ref{nonesti}.
Further, for any $\uf \in \mathcal{H}$ and $\tau \in [0,\infty)$  we set 
\begin{align*}
\Wf(\uf,\tau)(\rho):=&(T-T_0)^{2+\frac{2}{p-1}}e^{-(2+\frac{2}{p-1}) \tau}F\bigg(T-(T-T_0)e^{-\tau},(T-T_0) \rho e^{-\tau},\\
&(T-T_0)^{-\frac{2}{p-1}}e^{\frac{2}{p-1}\tau}(u_1+iu_3)(\rho),
\\&(T-T_0)^{-(1+\frac{2}{p-1})}e^{(1+\frac{2}{p-1})\tau}((u_2+iu_4)(\rho)-\frac{2}{p-1}(u_1+i u_3)(\rho)),\\&(T-T_0)^{-(1+\frac{2}{p-1})}e^{(1+\frac{2}{p-1})\tau}\partial_\rho (u_1+iu_3)(\rho)\bigg).
\end{align*}
and
\begin{align*}
\Vf(\uf,\tau)(\rho):=\begin{pmatrix}
0\\
\Re\left( \Wf(\uf,\tau)(\rho)\right)\\
0\\
\Im\left( \Wf(\uf,\tau)(\rho)\right)
\end{pmatrix},
\end{align*}
in accordance with the transformations from Section 1.
Lemma \ref{lemonV} shows that this really defines an operator mapping from $\mathcal{H}\times [0,\infty)$ to $\mathcal{H}$.
With these definitions, Lemma \ref{Semigroupgen} now enables us to abstractly rewrite Eq.~\eqref{eq2} as
\begin{equation}\label{abstraceq}
\partial_\tau \Psi(\tau)=\Lf \Psi(\tau)+ \Nf(\Psi(\tau))+\Vf(\Psi(\tau),\tau),
\end{equation}
where $\Psi$ is a function mapping from some interval $I \subset [0,\infty)$ that contains 0 to $\mathcal{H}$. Now we can also provide the aforementioned definition of a solution. 
\begin{defi}
	We call a function $u: \Gamma^T_{T_0}\rightarrow \C$ a solution of Eq.~\eqref{startingeq}, if the corresponding $\Psi:[0,\infty)\mapsto \mathcal{H}$ belongs to $C([0,\infty),\mathcal{H})$ and satisfies
	\begin{equation}
\Psi(\tau)=\Sf(\tau)\Psi(0)+ \int_0^{\tau}\Sf(\tau-\sigma)\left(\Nf(\Psi(\sigma))+\Vf(\Psi(\sigma),\sigma)\right)	d\sigma,\end{equation}
for all $\tau \geq 0$.
\end{defi} 
Next, let $\Psi_\theta$ be the function we obtain by applying the previously used transformations to the blowup function $\psi_\theta^T$. This yields
\begin{equation}
\Psi_\theta:=\begin{pmatrix}
\psi_{\theta,1}\\
\psi_{\theta,2}\\
\psi_{\theta,3}\\
\psi_{\theta,4}
\end{pmatrix}=
\begin{pmatrix}
&\kappa_p \cos(\theta) \\
&\frac{2}{p-1}\kappa_p\cos(\theta)\\
&\kappa_p \sin(\theta)\\
&\frac{2}{p-1}\kappa_p \sin(\theta)
\end{pmatrix}.
\end{equation}
Further, we let $\theta$ depend directly on $\tau$ and assume that $\lim_{\tau \to \infty}\theta(\tau)=: \theta_\infty$ exists.
As our goal is to study the behavior of solutions that are close to the family of blowup functions we make the ansatz
\begin{equation}\label{splitting ansatz}
\Psi(\tau)= \Phi(\tau)+\Psi_{\theta(\tau)}.
\end{equation}
By inserting this into Eq.~\eqref{abstraceq} and setting 

$$ 
\Lf_{\theta}'\uf (\rho)=\begin{pmatrix}
0 & 0& 0&0\\
 (p-1)c_p \cos(\theta)^2+c_p &0 & (p-1)c_p\cos(\theta)\sin(\theta) &0\\
 0 &0&0&0 \\
(p-1)c_p\cos(\theta)\sin(\theta) &0& (p-1)c_p\sin(\theta)^2+c_p &0
\end{pmatrix}
\begin{pmatrix}
u_1(\rho)\\
u_2(\rho)\\
u_3(\rho)\\
u_4(\rho)\\
\end{pmatrix}
$$
as well as
$$
\N_{\theta(\tau)}(\uf):=\N(\Psi_{\theta(\tau)}+\uf)-\N(\Psi_{\theta(\tau)})-\Lf_{\theta(\tau)}'\uf,
$$
we obtain the equation
\begin{align}\label{eqtosovle}
\partial_{\tau}\Phi(\tau)-\Lf\Phi(\tau)-\Lf_{\theta_{\infty}}'\Phi(\tau)=\left( \Lf_{\theta(\tau)}'-\Lf'_{\theta_{\infty}} \right)\Phi(\tau)\nonumber&+\N_{\theta(\tau)}(\Phi(\tau))
\\&+\Vf(\Phi(\tau)+\Psi_{\theta(\tau)}(\tau),\tau)-\partial_{\tau}\Psi_{\theta(\tau)}.
\end{align}
\begin{lem}\label{growth}
	For any $\theta \in \R$ the operator $\Lf_\theta:=\Lf+\Lf_\theta'$ generates a strongly continuous semigroup $\Sf_\theta:[0,\infty)\to \mathcal{B}(\mathcal{H})$.
\end{lem}
\begin{proof}
Since $\Lf_\theta'$ is a bounded linear operator on $\mathcal{H}$, the claim follows from the Bounded Perturbation Theorem.
\end{proof}
\subsection{Spectral Analysis of $\Lf_0$}
In order to proceed, it is essential to recover a growth estimate for $\Sf_\theta$. To this end, we will now compute the spectrum of $\Lf_0$ and then subsequently also $\sigma(\Lf_\theta)$ for $\theta$ small enough in absolute value.
But first, one more small preliminary lemma is required.

\begin{lem}\label{identity}
	Let $\theta \in \R$. Then $\lambda \in  \sigma(\Lf_\theta)\setminus\sigma(\Lf)$ implies $\lambda \in \sigma_p(\Lf_\theta)$. 
\end{lem}
\begin{proof}
	Let $\lambda \in  \sigma(\Lf_\theta)\setminus\sigma(\Lf)$.
	Then we have the identity
	$$
	\lambda-\Lf_{\theta} =(1-\Lf_{\theta}'\Rf_{\Lf}(\lambda))(\lambda-\Lf).
	$$
	Therefore, as $\Lf_{\theta}'\Rf_{\Lf}(\lambda)$ is a compact operator, we obtain $\lambda \in  \sigma_p(\Lf_\theta)$ by employing the spectral theorem for compact operators.
\end{proof}
This result enables us to explicitly calculate the spectrum in the case $\theta=0$. 
\begin{prop}
	The spectrum of $\Lf_0$ is contained in the set $$\{\lambda \in \C: \Re(\lambda)\leq -\frac{2}{p-1}\}\cup\{0,1\}.$$
\end{prop}

\begin{proof}
	The growth estimate given by Lemma \ref{Semigroupgen} implies that $\sigma(\Lf)$ is contained in the set $\{\lambda\in \C:\Re(\lambda)\leq -\frac{2}{p-1}\}.$ 
	Hence, any spectral point $\lambda$ with $\Re(\lambda)> -\frac{2}{p-1}$ has to be an eigenvalue by Lemma \ref{identity}. Therefore, there exists a nontrivial $\uf \in D(\Lf_0) $ with $(\lambda - \Lf_0)\uf=\textbf{0}$. From $(\lambda-\Lf_0)\uf=0$ we obtain
	\begin{align*}
	u_{j+1}(\rho)=(\lambda+\frac{2}{p-1}) u_j+\rho u_j'(\rho),
	\end{align*}
	for $j=1,3$.
	A direct calculation now shows that $(\lambda-\Lf_0)\uf=0$ implies
		\begin{align}
	\begin{cases}
	&-(1-\rho^2)u_1''(\rho) +2(\lambda \rho -\frac{1}{\rho}+ \frac{p+1}{p-1}\rho) u_1'(\rho)+(\lambda(\lambda+\frac{p+3}{p-1})-2\frac{p+1}{p-1})u_1(\rho)=0\\
	&-(1-\rho^2)u_3''(\rho) +2(\lambda\rho  -\frac{1}{\rho} +\frac{p+1}{p-1}\rho)u_3'(\rho)+ \lambda (\lambda+\frac{p+3}{p-1})u_3(\rho)=0.
	\end{cases}
	\end{align}
	By setting $u_1(\rho)=\frac{v_1(\rho)}{\rho}$ and $u_3(\rho)=\frac{v_2(\rho)}{\rho}$, this system turns into
	\begin{align}
	\begin{cases}
	&-(1-\rho^2)v_1''(\rho) +2(\lambda +  \frac{2}{p-1})\rho v_1'(\rho)+( (\lambda+\frac{2}{p-1})(\lambda+\frac{2}{p-1}-1)-pc_p)v_1(\rho)=0\\
	&-(1-\rho^2)v_2''(\rho) +2(\lambda +  \frac{2}{p-1})\rho v_2'(\rho)+( (\lambda+\frac{2}{p-1})(\lambda+\frac{2}{p-1}-1)-c_p)v_2(\rho)=0.
	\end{cases}
	\end{align}
	Since the two equations decouple, we will consider them separately.
	The first one has already been studied in Lemma 3.5 of \cite{DonSch12}, where the authors showed, with the help of hypergeometric functions, that the only eigenvalue of this equation is 1. In order to analyze the second one, we make the substitution $\rho \mapsto z=\rho^2$ and set $w(z)=v_2(\sqrt{z})$ to obtain
	$$
	z(1-z)w''(z)+\left(\frac{1}{2}-(\lambda+\frac{2}{p-1}+\frac{1}{2})z\right)w'(z)-\frac{1}{4}\left(\lambda^2-\frac{p-5}{p-1}\lambda-\frac{4}{p-1}\right)w(z)=0.
	$$
	Next, by setting $c=\frac{1}{2},\,a=\frac{1}{2}(\lambda-1)$ and $ b=\frac{1}{2}(\lambda+\frac{4}{p-1})$, the equation turns into
	\begin{equation}
	z(1-z)w''(z)+(c-(a+b+1)z)w'(z)-ab\,w(z)=0.
	\end{equation}
	Around $z=0$ a fundamental system of solutions is given by
	\begin{align*}
	g_1(z)&=\, _2F_1(a,b;c;z)\\
	g_2(z)&=\,z^\frac{1}{2}\, _2F_1(a-c+1,b-c+1;2-c;z)
	\end{align*}
	where $ _2F_1$ denotes the standard hypergeometric function (see for instance \cite{OlvLonBoiClar10}). If $c-a-b$ does not vanish, a fundamental system around $z=1$ is given by 
	\begin{align*}
	f_1(z)&= \,_2F_1(a,b;a+b+1-c;1-z)\\
	f_2(z)&=(1-z)^{c-a-b}\, _2F_1(c-a,c-b;c-a-b+1;1-z).
	\end{align*}
	If $(c-a-b)=0$, a fundamental system around $z=1$ is given by $f_1$ and a second solution which diverges logarithmically for $z\rightarrow 1$.
	Since $\Re(\lambda)$ is assumed to be bigger than $-\frac{2}{p-1},$ $f_2 \notin H^2(\B_1^3)$  and hence for a solution to be in  $H^2(\B_1^3)$, it must be a multiple of $f_1$. Therefore, there have to be constants $c_1$ and $c_2$ such that $f_1 = c_1 g_1+ c_2 g_2$. Since the solution has to satisfy the boundary condition $w(0)=0$, which stems from the transformation $\rho u_j(\rho)=v_j(\rho)$ and the fact that we require $u_j\in H^2(\B_1^3)$, the coefficient  $c_1$ has to vanish. Thanks to the explicit corresponding connection formula, the coefficient is given by 
	$$
	c_1=\dfrac{\Gamma(a+b+1-c)\Gamma(1-c)}{\Gamma(a+1-c)\Gamma(b+1-c)},
	$$
	where $\Gamma$ denotes the gamma function.
	For $c_1$ to vanish, $a+1-c$ or $b+1-c$ needs to be a pole of $\Gamma$. This yields $\lambda = -2k $ for $k \in \mathbb{N}_0$   or $\lambda= -1-k-\frac{4}{p-1}$ for $ k\in \mathbb{N}_0$. Therefore $\lambda$ has to be real and since $\Re(\lambda)$ is assumed to be bigger than $-\frac{2}{p-1}$, the only possible choice is 0.
\end{proof}
The next two lemmas will determine the corresponding geometric and algebraic multiplicities of the eigenvalues $0$ and $1$.
\begin{lem}
	The eigenvalues 0 and 1 both have geometric multiplicity 1. Furthermore the  geometric eigenspace corresponding to the eigenvalue $0$ is spanned by $\rf_0$ while the other one is spanned by $
	\gf_0,
	$
	with
	$$\rf_0(\rho):=
	\begin{pmatrix}
	0\\
	0\\
	1\\
	\frac{2}{p-1}
	\end{pmatrix}
	\;\; and \quad 
	\gf_0(\rho):=
	\begin{pmatrix}
	1\\
	\frac{p+1}{p-1}\\
	0\\
	0
	\end{pmatrix}.
	$$
\end{lem}
\begin{proof}
	First of all, it is straightforward to check that both functions are indeed eigenfunctions to the corresponding eigenvalues. 
	As before, the eigenvalue $1$ was already dealt with in \cite{DonSch12}, Lemma 3.6 and therefore we will only do the considerations for the eigenvalue 0. Assume that there is another eigenfunction $\tilde{\rf}(\rho)$. This would then imply that $\tilde{r}_4(\rho)=\frac{2}{p-1}\tilde{r}_3(\rho)+\rho \tilde{r}_3(\rho)$ and  that $\tilde{r}(\rho)=\rho \tilde{r}_3$
	satisfies \begin{equation}\label{eqnootherev}
	-(1-\rho^2)\tilde{r}''(\rho)+\frac{4}{p-1}\rho\tilde{r}'(\rho)-\frac{4}{p-1}\tilde{r}(\rho)=0.
	\end{equation}
	A fundamental system of solutions for this equation is given by
	\begin{align*}
	f_1(\rho)&=\rho,\\
	f_2(\rho)&=(1-\rho^2)^{1-\frac{2}{p-1}}\, _2F_1(1,\frac{1}{2}-\frac{2}{p-1};2-\frac{2}{p-1};1-\rho^2).
	\end{align*}
	Now any solution of Eq.~\eqref{eqnootherev} has to be a linear combination of these two solutions. But as $\tilde{r}$ has to be an element of $H^2(\B_1^3)$, which $f_2$ is not, it has to be a multiple of $f_1$. Therefore $ \tilde{r}_3=c$ for some $c \in  \C$ and thus, by the above expression for $\tilde{r}_4$ the claim follows.
\end{proof}
Next, we define the Riesz projections corresponding to the eigenvalues 0 and 1 via
\begin{align*}
\Po= \frac{1}{2\pi i}\int_{\gamma_0} \Rf_{\Lf_0}(z) \, dz,\\
\Qo= \frac{1}{2\pi i}\int_{\gamma_1} \Rf_{\Lf_0}(z) \, dz
\end{align*}
where the two curves $\gamma_j$ map from $[0,1]$ to $\C$ and are defined  by
\begin{align}\label{curves}
\gamma_0(t)= d e^{2\pi i t} \, \text{ and } \gamma_1(t)= 1+\frac{1}{2}e^{2 \pi i t}.
\end{align}
Here $d$ is chosen small enough, such that the curve $\gamma_0$ stays completely in the resolvent set of $\Lf_0$. A suitable choice for $d$ would for instance be $$d=\frac{1}{p-1}.$$
Further, we define the subspaces $ \Mf_0:= \Po \mathcal{H}$, $\Mf_1:= \Qo\mathcal{H}$ and $\Nf:= (\I-\Po-\Qo)\mathcal{H}$.
With these definitions at hand, the next lemma can be shown.
\begin{lem}\label{spectrumL0}
	The projections $\Po$ and $\Qo$ both have rank 1 and the subspaces $\Mf_0$ and $\Mf_1$ are spanned by $\gf_0$ and $\rf_0$, respectively.
\end{lem}
\begin{proof}
	Note that since both eigenvalues are not in the spectrum of $\Lf$, they have to be generated by the compact perturbation $\Lf'_0$. This implies that the dimensions of both eigenspaces have to be finite, as they would otherwise be in the essential spectrum, which is stable under compact perturbations (see \cite{Kat} p. 244, Theorem 5.35). Since the same arguments as in Lemma 3.7 of \cite{DonSch12} apply here as well, only the claim for $\Mf_0$ needs to be established.
	Note that the operator $\Lf_{\Mf_0}$ defined by $\Lf_{\Mf_0}\uf=\Lf_0\uf$ with $D(\Lf_{\Mf_0})$ = $D(\Lf) \cap \Mf_0$ can be regarded as an operator mapping from the closed subspace $\Mf_0$ to $\Mf_0$. Furthermore the spectrum of this operator only consists of the point $0$. Since the inclusion \Span$\{\rf_0\} \subset \Mf_0$ is immediate, only the reverse inclusion remains to be shown. To see this, note that $\Lf_{\Mf_0}$ is nilpotent as its only eigenvalue is $0$. Thus, there exists a minimal $n\in \mathbb{N}$ such that $\Lf_{\Mf_0}^n\uf=\textbf{0}$, for all $\uf\in \Mf_0$. If $n=1$, then $\Mf_0 \subset $ \Span$\{ \rf_0 \}$ and there is nothing to show. If $ n\geq 2$, then there is a nontrivial $\vf \in \rg \Lf_{\Mf_0}$ with $\Lf_{\Mf_0} \vf=\textbf{0}$. Since this forces $\vf$ to be a multiple of $\rf_0$, there exists a $ \uf \in D(\Lf_{\Mf_0})$ with $\Lf_{\Mf_0}\uf = c \rf_0$. This implies that $\tilde{u}(\rho) = \rho u_3 (\rho)$ satisfies
	\begin{equation}
	-(1-\rho^2)\tilde{u}''+\frac{4}{p-1}\rho \tilde{u}'(\rho)-\frac{4}{p-1}\tilde{u}(\rho)=\frac{p+3}{p-1}\rho=:R(\rho).
	\end{equation}
	As before, a fundamental system for the homogeneous equation is given by
	\begin{align*}
	f_1(\rho)&=\rho\\
	f_2(\rho)&=(1-\rho^2)^{1-\frac{2}{p-1}} \, _2F_1(1,\frac{1}{2}-\frac{2}{p-1};2-\frac{2}{p-1};1-\rho^2).
	\end{align*}
	The Wronskian of these two functions is given by$$
	W\left(f_1,f_2\right)(\rho)= c(1-\rho^2)^{-\frac{2}{p-1}}
	$$
	where $c\neq 0$ is some constant.
	Thus, a solution to the inhomogeneous equation must be of the form
	\begin{equation}
	\tilde{u}(\rho)=c_1 f_1+c_2 f_2-\frac{1}{c} f_1(\rho)\int_{\rho_1}^\rho f_2(\tilde{\rho}) R(\tilde{\rho}) (1-\tilde{\rho}^2)^{\frac{3-p}{p-1}} \, d\tilde{\rho} +\frac{1}{c} f_2(\rho)\int_{\rho_2}^\rho f_1(\tilde{\rho}) R(\tilde{\rho})(1-\tilde{\rho}^2)^{\frac{3-p}{p-1}}\, d\tilde{\rho}
	\end{equation}
	for some constants $c_1, c_2 \in \C$ and $\rho_0, \rho_1 \in [0,1]$.
	The boundary condition $\tilde{u}(0)=0$ implies that $c_2= -\frac{1}{c}\int_{\rho_2}^0 f_1(\tilde{\rho}) R(\tilde{\rho})(1-\tilde{\rho}^2)^{\frac{3-p}{p-1}}\, d\tilde{\rho}.
	$
	Plugging this into the equation yields
	$$
	\tilde{u}(\rho)=\left(c_1-\frac{1}{c}\int_{\rho_1}^\rho f_2(\tilde{\rho}) R(\tilde{\rho}) (1-\tilde{\rho}^2)^{\frac{3-p}{p-1}} \, d\tilde{\rho}\right)\rho- \frac{1}{c} f_2(\rho) \int_0^\rho f_1(\tilde{\rho}) R(\tilde{\rho})(1-\tilde{\rho}^2)^{\frac{3-p}{p-1}}\, d\tilde{\rho}.
	$$
	But since $f_2$ is not in $H^2(\B_1^3)$, the integral$$
	\int_0^1 f_1(\tilde{\rho}) R(\tilde{\rho})(1-\tilde{\rho}^2)^{\frac{3-p}{p-1}}\, d\tilde{\rho}$$
	would have to vanish. This is however impossible as the integrand is strictly positive on $(0,1)$.
\end{proof}
Having sufficiently well characterized the spectrum of $\Lf_0$, we now turn to the spectrum of $\Lf_\theta$.
\subsection{Spectrum of $\Lf_\theta$}
The first easy to establish Lemma which we are going to need is the following.
\begin{lem}\label{liplem1}
	The operator $\Lf_\theta'$ is Lipschitz with respect to $\theta$, i.e.,
	$$
	\|\Lf_{\theta_1}'-\Lf_{\theta_2}'\| \lesssim |\theta_1-\theta_2|
	$$
	for all $\theta_1,$ $\theta_2 \in \R$.
\end{lem}
\begin{proof}
	This is immediate, since all the expressions of $\Lf'_\theta$ which depend on $\theta$ are Lipschitz.
\end{proof}

The next lemma provides a first description of the resolvent set $\rho(\Lf_\theta)$. 
\begin{lem}\label{reslem2}
	There exists a $\delta >0$ such that any $\lambda \in \rho(\Lf_0)$ is also contained in $ \rho(\Lf_\theta)$, provided that $\theta$ satisfies $|\theta|\leq \delta \min \{1, \|\Rf_{\Lf_0}(\lambda)\|^{-1}\}.$ 
\end{lem}

\begin{proof}
	
	Let $\lambda \in \rho(\Lf_0)$.
	Then the identity $\lambda-\Lf_\theta=\left[1+(\Lf_0'-\Lf_\theta')\Rf_{\Lf_0}(\lambda) \right](\lambda-\Lf_0)$ implies that $\lambda$ is in the resolvent set of $\Lf_\theta$ if and only if $1+(\Lf_0'-\Lf_\theta')\Rf_{\Lf_0}(\lambda)$ is bounded invertible. Since an explicit inverse can be given by the corresponding Neumann series, this expression is definitely bounded invertible if $\|\Lf_0'-\Lf_\theta'\| \|\Rf_{\Lf_0}(\lambda)\|<1$. Note that by Lemma \ref{liplem1} we have $\|\Lf_0'-\Lf_\theta'\| \leq L|\theta|$, for all $\theta\in \R$ and some fixed constant $L \in \R$. Hence, if we set $\delta < \frac{1}{2L}$, we obtain that $|\theta|< \delta\min\{1,\|\Rf_{\Lf_0}(\lambda)\|^{-1}\}$ yields the bounded invertibility of $1+(\Lf_0'-\Lf_\theta')\Rf_{\Lf_0}(\lambda)$. This in turn implies that $\lambda \in \rho(\Lf_\theta)$ and thus the proof is finished.
\end{proof}
As a next step, we define the two domains $$\Omega_{x_0, y_0}:=\{ z\in \C: \Re(z)\in [-\frac{3}{2(p-1)},x_0], \Im(z)\in[-y_0,y_0]\}$$
and
$$\Omega_{x_0,y_0}':=\{z\in \C:\Re(z)\geq -\frac{3}{2(p-1)}\} \setminus \Omega_{x_0,y_0}.$$
The following lemma will restrict possible eigenvalues of $\Lf_\theta$ to  a compact domain.
\begin{lem}
	There exist $x_0,y_0,\delta,c>0$, zero such that $\Omega'_{x_0,y_0} \subset\rho(\Lf_\theta)$ and
	$$\|\Rf_{\Lf_\theta}(\lambda)\|\leq c
	$$
	for all $\theta$ with $|\theta|\leq \delta$ and all $\lambda \in \Omega'_{x_0,y_0}$.
\end{lem}
\begin{proof}
	Let $\lambda \in \Omega_{x_0,y_0}'$. Then $\lambda$ is also in the resolvent set of $\Lf$ and one has the identity 
	$$
	\lambda-\Lf_{\theta}=\left(1-\Lf_{\theta}'\Rf_\Lf(\lambda)\right)(\lambda-\Lf).$$ Next, we claim that the estimate $\|\Lf_\theta'\Rf_\Lf(\lambda)\| \lesssim |\lambda|^{-1}$ holds true for all $\lambda \in \Omega_{x_0,y_0}'$,provided $x_0, y_0$ are chosen big enough and $\theta$ satisfies $|\theta|< \|\Rf_{\Lf_0}(\lambda)\|^{-1}$.
	Now, for any $\ff \in \mathcal{H}$ the expression $\Lf_\theta' \Rf_\Lf(\lambda) \ff$ written out explicitly reads as 
$$ 
\Lf_{\theta}'\Rf_{\Lf}(\lambda)\ff (\rho)=\begin{pmatrix}
0 & 0& 0&0\\
(p-1)c_p \cos(\theta)^2+c_p &0 & (p-1)c_p\cos(\theta)\sin(\theta) &0\\
0 &0&0&0 \\
(p-1)c_p\cos(\theta)\sin(\theta) &0& (p-1)c_p\sin(\theta)^2+c_p &0
\end{pmatrix}
\begin{pmatrix}
[\Rf_\Lf(\lambda)\ff]_1\\
[\Rf_\Lf(\lambda)\ff]_2\\
[\Rf_\Lf(\lambda)\ff]_3\\
[\Rf_\Lf(\lambda)\ff]_4\\
\end{pmatrix}
$$
Set $\uf=\Rf_{\Lf}(\lambda)\ff.$ Then $\uf$ is the unique solution of the equation $(\lambda-\Lf)\uf=\ff.$
Again, short calculation yields 
$$
u_{j+1}(\rho)=(\lambda+\frac{2}{p-1})u_j(\rho)+\rho u_j'(\rho)+f_j(\rho),
$$
for $j=1,3$
and hence
$$
\|u_j\|_{H^1(\B^3_1)}\lesssim |\lambda|^{-1}\left(\|u_j\|_{H^2(\B^3_1)}+\|u_{j+1}\|_{H^1(\B^3_1)}+\|f_j\|_{H^1(\B^3_1)}
\right).
$$
	Therefore, Lemma \ref{resbound} yields
	\begin{align*}
	\|[\Lf_\theta'\Rf_\Lf(\lambda)\ff]_{j+1}\|_{H^1(\B^3_1)} \lesssim&
	|\lambda|^{-1}\left(\|\Rf_\Lf(\lambda)\ff\|+\|\ff\|
	\right)\\
	\lesssim& |\lambda|^{-1}\|\ff\|,
	\end{align*}
	for $j=1,3$.
	Thus, the Neumann series $\sum_{k=0}^{\infty}\left(\Lf_\theta'\Rf_\Lf(\lambda)\right)^k$ converges and is uniformly bounded on $\Omega_{x_0,y_0}'$, if $x_0$ and $y_0$ are chosen sufficiently large, which in turn completes the proof.
\end{proof}
These results now enable us to describe the spectrum of $\Lf_\theta$ for small $\theta$.
\begin{lem}
	Let $\theta$ be sufficiently small. Then $$\sigma(\Lf_\theta)\subset \{z\in \C: \Re(z)\leq -\frac{3}{2(p-1)}\} \cup\{0,1\}.$$
	and $\{0,1\}\subset \sigma_p \left(\Lf_\theta\right)$.
	Furthermore, the eigenvalues $0$ and $1$ are simple.
	Finally, the eigenspaces of $0$ and $1$ are spanned by
	$$
	\rf_\theta(\rho):= \begin{pmatrix}
	-\kappa_p\sin(\theta)\\ -\kappa_p\sin(\theta)\frac{2}{p-1}\\ 
	\kappa_p\cos(\theta)\\ \kappa_p\cos(\theta)\frac{2}{p-1}
	\end{pmatrix} \text{ and }\;\,\,
	\gf_\theta(\rho):=\begin{pmatrix}
	\cos(\theta)\\ \cos(\theta)\frac{p+1}{p-1}\\
	\sin(\theta)\\ \sin(\theta)\frac{p+1}{p-1}
	\end{pmatrix},$$
	respectively.
\end{lem}
\begin{proof}
	Choose $x_0$ and $y_0$ large enough, such that $\overline{\Omega'_{x_0,y_0}}$ is contained in the resolvent set of $\Lf_\theta$ and set $M:= \max\{1,\sup_{z\in \partial \Omega_{x_0,y_0}}\Rf_{\Lf_0}(z) \}$. Note that by Lemma \ref{reslem2} one has that $|\theta| < \frac{\delta}{M}$ implies $\partial \Omega_{x_0,y_0} \subset \rho(\Lf_\theta)$, provided $\delta$ is chosen sufficiently small. We define the Projection $\Pf^{\mathrm{tot}}_\theta$ by
	$$
	\Pf^{\mathrm{tot}}_\theta=\frac{1}{2\pi i} \int_{\partial \Omega_{x_0,y_0}} \Rf_{\Lf_\theta}(z) dz.$$ Next, as an immediate consequence of the formula $\Rf_{\Lf_\theta}(\lambda)= \Rf_\Lf(\lambda)\left(\I-\Lf_\theta'\Rf_{\Lf}(\lambda)\right)^{-1},$ one has that $\Pf^{\mathrm{tot}}_\theta$ depends continuously on $\theta$. By Lemma \ref{spectrumL0}, $\Pf^{\mathrm{tot}}_0$ has rank 2 and therefore, Lemma $4.10$ of \cite{Kat} implies that $\Pro^{\mathrm{tot}}_\theta$ also has rank 2 for $\theta$ sufficiently small. Note that $0$ and $1$ are eigenvalues of $\Lf_\theta$ with corresponding eigenfunctions $\rf_\theta$ and $\gf_\theta$. Since the rank of $\Pro^{\mathrm{tot}}_\theta$ gives an upper bound on the sum of the geometric multiplicities of the eigenvalues, there can be no other eigenvalues in  $\Omega_{x_0,y_0}$ and the claim follows.
\end{proof}
\begin{prop}\label{projections}
	Let $\theta \in \R$ have a sufficiently small modulus. Then there exist two rank one projections $\Pf_\theta$, $\Qf_\theta \in \mathcal{B}(\mathcal{H})$ such that
	\begin{align*}
	[\Sf_\theta(\tau),\Pf_\theta]&=[\Sf_\theta(\tau),\Qf_\theta]=0
	\end{align*}
	and
	\begin{align*}
	\Pf_\theta \Qf_\theta &=\Qf_\theta \Pf_\theta =0,
	\end{align*} 
	where $[(.),(.)]$ denotes the commutator. These projections also satisfy
	\begin{align*}
	&\Sf_\theta(\tau)\Pf_\theta= \Pf_\theta\\
	&\Sf_\theta(\tau)\Qf_\theta=e^\tau \Qf_\theta
	\end{align*}
	and 
	\begin{align*}
	& \rg \Pf_\theta =\Span\{ \rf_\theta \}\\
	& \rg \Qf_\theta= \Span\{\gf_\theta \},
	\end{align*}
	for all $\tau \geq 0$.
	Furthermore, define $\tilde{\Pf}_\theta$ as $\tilde{\Pf}_\theta:=\I-\Pf_\theta-\Qf_\theta$. Then one has the bound
	$$
	\|\Sf_\theta(\tau)\tilde{\Pf}_\theta \uf\|\lesssim e^{-\frac{4}{3(p-1)} \tau}\|\tilde{\Pf}_\theta \uf\|
	$$
	for all $\uf \in \mathcal{H}$ and all $ \tau \geq 0$ and all $\theta \in \R$ with $|\theta|$ small enough.
\end{prop}
\begin{proof}
	Analogously to $\Pf_0$ and $\Qf_0$, we define the spectral projections $\Pf_\theta$ and $\Qf_\theta$ by
	$$
	\Pf_\theta:= \frac{1}{2 \pi i }\int_{\gamma_0} \Rf_{\Lf_\theta}(z) dz \text{ and } \Qf_\theta :=\frac{1}{2 \pi i }\int_{\gamma_1} \Rf_{\Lf_\theta}(z) dz
	$$
	with $\gamma_0$ and $\gamma_1$ defined as in (\ref{curves}). 
	As $\Pf_\theta$ and $\Qf_\theta$ are spectral projections, it follows that $\Pf_\theta \Qf_\theta =\Qf_\theta \Pf_\theta =0$. Since the operator $\Lf_\theta$ commutes with each of the two projections, also the semigroup it generates does so. Finally, to establish the growth estimate on $\Sf_\theta(\tau)\tilde{\Pf}_\theta$ note that one has 
	$ \sup_{z \in \mathbb{H}_{p}^{+}}\|\Rf_\Lf(z)\tilde{\Pf}_\theta\| < \infty$, where $\mathbb{H}_{p}^{+}:= \{ z\in \C : \Re(z) \geq -\frac{4}{3(p-1)}\}$. Therefore we can apply the Gearhart-Prüss-Greiner Theorem (see \cite{EngNag99}, p. 302, Theorem 1.11) to obtain the estimate and hence conclude the proof of this proposition, as the rest follows from the previous lemmas.
\end{proof}
The final result of this section are three more Lipschitz estimates that will be essential later on.
\begin{lem}\label{liponrg}
	We have
	\begin{align*}
	\|\gf_{\theta_1}-\gf_{\theta_2}\|+\|\rf_{\theta_1}-\rf_{\theta_2}\|&\lesssim |\theta_1-\theta_2|\\
	\|\Pf_{\theta_1}-\Pf_{\theta_2}\|+\|\Qf_{\theta_1}-\Qf_{\theta_2}\| &\lesssim|\theta_1-\theta_2|\\
	\|\Sf_{\theta_1}(\tau)\tilde{\Pf}_{\theta_1}-\Sf_{\theta_2}(\tau)\tilde{\Pf}_{\theta_2}\|&\lesssim |\theta_1-\theta_2|e^{-\frac{1}{p-1}\tau},
	\end{align*}
for all $\theta_1,\theta_2\in \R$ with $|\theta_1|,|\theta_2|$ sufficiently small and all $\tau \geq 0$. 
\end{lem}

\begin{proof}
	The estimate on $\gf_\theta$ and $ \rf_\theta$ follows from the fundamental theorem of calculus, as both functions are smooth with respect to $\theta$. The second and third estimate follow from  the proof of Lemma $4.9$ in \cite{DonSch16}.
\end{proof}
With this we conclude the linear analysis of Eq.~\eqref{abstraceq} and move on to the nonlinear part.
\section{Nonlinear theory}

In this section we will now deal with the nonlinearity $\N_{\theta(\tau)}$, which was defined as $$\N_{\theta(\tau)}(\uf)=\N(\Psi_{\theta(\tau)}+\uf)-\N(\Psi_{\theta(\tau)})-\Lf_{\theta(\tau)}'\uf$$
where
$$
\Nf(\uf)(\rho):=\begin{pmatrix}
0\\
\left|\left( u_1(\rho)\, ,u_3(\rho)\right)\right|^{p-1}u_1(\rho)\\
0\\
\left|\left( u_1(\rho)\, ,u_3(\rho)\right)\right|^{p-1}u_3(\rho)\\
\end{pmatrix},
$$
and with 
\begin{align*}
\Vf(\uf,\tau)(\rho):=\begin{pmatrix}
0\\
\Re\left( \Wf(\uf,\tau)(\rho)\right)\\
0\\
\Im\left( \Wf(\uf,\tau)(\rho)\right)
\end{pmatrix},
\end{align*}
whereas $\Wf$ was defined as
\begin{align*}
\Wf(\uf,\tau)(\rho):=&(T-T_0)^{2+\frac{2}{p-1}}e^{-(2+\frac{2}{p-1}) \tau}F\bigg(T-(T-T_0)e^{-\tau},(T-T_0) \rho e^{-\tau},\\
&(T-T_0)^{-\frac{2}{p-1}}e^{\frac{2}{p-1}\tau}(u_1+iu_3)(\rho),
\\&(T-T_0)^{-(1+\frac{2}{p-1})}e^{(1+\frac{2}{p-1})\tau}((u_2+iu_4)(\rho)-\frac{2}{p-1}(u_1+i u_3)(\rho)),\\&(T-T_0)^{-(1+\frac{2}{p-1})}e^{(1+\frac{2}{p-1})\tau}\partial_\rho (u_1+iu_3)(\rho)\bigg)
\end{align*}
for any $\uf \in \mathcal{H}$.
We also recall that in the similarity coordinates, which we use, the static blowup function takes the form
\begin{align*}
\Psi_\theta=
\begin{pmatrix}
&\kappa_p \cos(\theta) \\
&\frac{2}{p-1}\kappa_p\cos(\theta)\\
&\kappa_p \sin(\theta)\\
&\frac{2}{p-1}\kappa_p \sin(\theta)
\end{pmatrix}.
\end{align*}
\subsection{Estimates on the nonlinearity}
The first important estimate of the nonlinear theory is the following.
\begin{lem}\label{nonesti}
There exists a $\delta>0$ such that
\begin{equation}
\|\N_{\theta_1}(\uf)-\N_{\theta_2}(\vf)\|\lesssim \left(\|\uf\|+\|\vf\|\right)\|\uf-\vf\|+\left(\|\uf\|^2+\|\vf\|^2\right)|\theta_1-\theta_2|,
\end{equation}
for any $\uf,\vf \in \mathcal{H}$ with $\|\uf\|,\|\vf\|<\delta$ and $\theta_1,\theta_2\in \R$ with $|\theta_1|,|\theta_2|<\delta$. 
Furthermore $\N_\theta(\textup{\textbf{0}})=\textup{\textbf{0}}.$
\end{lem}
\begin{proof}
As $\Psi_{\theta}$ is independent of $\rho$ and smooth as a function of $\theta$, this follows analogously to Lemma 5.2 in \cite{DonSch16}.
\end{proof}
Next, we will prove a similar result for $\Vf$. In order to do that, we recall that the perturbation $F$ is of the form
\begin{equation*}
F(t,r,u,v,w)= A(t,r,u)+B(t,r,u)v+C(t,r,u)w,
\end{equation*}
where $A$, $B$, and $C$ satisfy
\begin{align*}
|A(t,r,u)|&\leq M(1+|u|^{q})
\\
|A(t,r,u_1)-A(t,r,u_2)|
&\leq M\left|u_1|u_1|^{q-1}-u_2|u_2|^{q-1}\right|
\\
|B(t,r,u)|+|C(t,r,u)|&\leq M(1+ |u|)
\\
|B(t,r,u_1)-B(t,r,u_2)|+|C(t,r,u_1)-C(t,r,u_2)|&\leq M |u_1-u_2|,
\end{align*}
while the whole perturbation $F$ satisfies
\begin{align*}
|\partial_r F(t,r,u,v,w)|\leq& M(1+|u|^{q}+(1+|u|)(|v|+|w|))\\
|\partial_{x} F(t,r,x+ iy,v,w)|+|\partial_{y} F(t,r,x + iy,v,w)|\leq& M(1+|u|^{q-1}+|v|+|w|).
\end{align*}
as well as
\begin{align*}
&|\partial_rF(t,r,u_1,v_1,w_1)-\partial_r F(t,r,u_2,v_2,w_2)|
\\
&\leq M \big( \left|u_1|u_1|^{q-1}-u_2|u_2|^{q-1} \right|
+ |v_1-v_2|+|w_1-w_2|\\
&\;\;\;+|u_1v_1-u_2v_2|+|u_1w_1-u_2w_2|\big)
\\
&|\partial_{x_1}F(t,r,x_1+i y_1,v_1,w_1)- \partial_{x_2}F(t,r,x_2+iy_2,v_2,w_2)|
\\
&\leq M\left(\left| u_1|u_1|^{q-2}-u_2|u_2|^{q-2}\right|+|v_1-v_2|+|w_1-w_2|\right)
\\
&|\partial_{y_1}F(t,r,x_1+i y_1,v_1,w_1)- \partial_{y_2}F(t,r,x_2+iy_2,v_2,w_2)|
\\
&\leq M\left(\left|u_1|u_1|^{q-2}-u_2|u_2|^{q-2}\right|+|v_1-v_2|+|w_1-w_2|\right),
\end{align*}
for $1\leq q<p $ and $M>0$.
Note that these estimates imply
\begin{align*}
|\partial_{x_1}B(t,r,x_1+ iy_1)-\partial_{x_2}B(t,r,x_2+ iy_2)|+|\partial_{x_1}C(t,r,x_1+ i y_1)-\partial_{x_2}C(t,r,x_2+ i y_2)|&=0\\
|\partial_{y_1}B(t,r,x_1+ iy_1)-\partial_{y_2}B(t,r,x_2+ iy_2)|+|\partial_{y_1}C(t,r,x_1+ i y_1)-\partial_{y_2}C(t,r,x_2+ i y_2)|&=0.
\end{align*}

\begin{lem}\label{lemonV}
The operator $\Vf$ maps $\mathcal{H}\times [0,\infty)$ to $\mathcal{H}$ and there exists a $\tilde{q}>0$ such that
\begin{align*}
\|\Vf(\uf,\tau)\|&\lesssim (T-T_0)^{\tilde{q}} e^{-\tilde{q}\tau}\left(1+\|\uf\|^{q}+\|\uf\|^2\right)
\end{align*}
and
\begin{align*}
\|\Vf(\uf,\tau)-\Vf(\vf,\tau)\|&\lesssim (T-T_0)^{\tilde{q}} e^{-\tilde{q}\tau}\|\uf-\vf\|\Big(1+\|\uf\|+\|\vf\|+\|\uf\|^{q-1}+\|\vf\|^{q-1}
\\
& +\|\uf\|\left(\left\||u_1+iu_3|^{q-2}\right\|_{L^\infty(\B^3_1)}+\left\||v_1+iv_3|^{q-2}\right\|_{L^\infty(\B^3_1)}\right)\Big),
\end{align*}
for any $\uf, \vf \in \mathcal{H}$ and $ \tau \in [0,\infty).$
\end{lem}

\begin{proof}
Set
$$
\tilde{q}=2+\frac{2}{p-1}-\max\{q \frac{2}{p-1},1+\frac{4}{p-1}\}
$$
and note that $ q<p$ implies $\tilde{q}>0$.
Therefore, the bounds on $F,$ the Banach algebra property of $H^2(\B^3_1)$, and Hölder's inequality yield the estimates
\begin{align*}
\left\|[\Vf(\uf,\tau)(\rho)]_k\right\|_{L^2(\B^3_1)}\lesssim& (T-T_0)^{\tilde{q}} e^{-\tilde{q}\tau}\bigg(1+\sum_{j=1,3}\Big( \||u_j|^q\|_{L^2(\B^3_1)}+\|(|u_1+ iu_3|+1) u_j'\|_{L^2(\B^3_1)}
\\
&+\|(|u_1+ iu_3|+1)(|u_j|+|u_{j+1}|)\|_{L^2(\B^3_1)}\Big)
\bigg)
\\
\lesssim& (T-T_0)^{\tilde{q}}e^{-\tilde{q}\tau}\left(1+\|\uf\|^q+\|\uf\|^2+\|\uf\|\right)
\end{align*}
and
\begin{align*}
\left\|[\Vf(\uf,\tau)]_k\right\|_{\dot{H}^1(\B^3_1)}\lesssim& (T-T_0)^{\tilde{q}} e^{-\tilde{q}\tau}\bigg(1+\sum_{j=1,3} \Big[\||u_j|^q\|_{L^2(\B^3_1)}+\|(|u_1+ iu_3|+1)u_j'\|_{L^2(\B^3_1)}
\\
&+\|(|u_1+ iu_3|+1)(|u_j|+|u_{j+1}|)\|_{L^2(\B^3_1)}
\bigg)
\\
&+\| u'_j\|_{L^4(\B^3_1)} \Big(1+\||u_j|^{q-1}\|_{L^4(\B^3_1)} +\|u_j'\|_{L^4(\B^3_1)}
+\|u_{j}\|_{L^4(\B^3_1)}
\\
&+\|u_{j+1}\|_{L^4(\B^3_1)}\Big)
\\
 &+\sum_{l=1}^4\Big(\|(u_1'+ iu'_{3})u_l\|_{L^2(\B^3_1)}
 +\|(1+|u_1+ iu_{3}|)u_l'\|_{L^2(\B^3_1)}\Big)
 \\
 &+ \|(u'_j)^2\|_{L^2(\B^3_1)}+\|(1+|u_1+iu_3|)u''_j\|_{L^2(\B^3_1)} \Big]
 \bigg)\\
\lesssim& (T-T_0)^{\tilde{q}}e^{-\tilde{q}\tau}\left(1+\|\uf\|^q+\|\uf\|^2+\|\uf\|+\sum_{j=1,3}\|u'_j\|^2_{L^4(\B^3_1)}\right),
\end{align*}
for $k=2,4$, $\uf \in \mathcal{H},$ and $\tau \geq0$.
Thus, the estimate
\begin{align*}
\|\Vf(\uf,\tau)\|&\lesssim (T-T_0)^{\tilde{q}} e^{-\tilde{q}\tau}\left(1+\|\uf\|+\|\uf\|^{q}+\|\uf\|^2\right)
\end{align*}
follows from the Sobolev inequality $\|.\|_{L^4(\B^3_1)}\lesssim \|.\|_{H^1(\B^3_1)}$. The first estimate stated in the Lemma now follows from the elementary inequality
$$
\|\uf\|\leq1+\|\uf\|^2.
$$
Similarly, the bounds on the perturbation imply
\begin{align*}
\|&[\Vf(\uf,\tau)-\Vf(\vf,\tau)]_k\|_{L^2(\B^3_1)}\\
\lesssim& (T-T_0)^{\tilde{q}} e^{-\tilde{q}\tau}\bigg(\left\|(u_1+i u_3)|u_1+iu_3|^{q-1}-(v_1+iv_3)|v_1+iv_3|^{q-1}\right\|_{L^2(\B^3_1)}
\\
&+\sum_{j=1,3}\Big(\|(1+|u_1+ i u_3|)(u'_j- v'_j)\|_{L^2(\B^3_1)}+ \| v'_j(u_1+i u_3-(v_1+i v_3))\|_{L^2(\B^3_1)}
\\
&+\|(1+|u_1+ i u_3|)(u_{j+1}- v_{j+1})\|_{L^2(\B^3_1)}+ \left\| v_{j+1}(u_1+i u_3-(v_1+i v_3))\right\|_{L^2(\B^3_1)}\Big)\bigg)\\
\lesssim& (T-T_0)^{\tilde{q}} e^{-\tilde{q}\tau}\|\uf-\vf\|\left(1+\|\uf\|^{q-1}+\|\vf\|^{q-1}+\|\uf\|+\|\vf\|\right),
\end{align*}
for $k=2,4$, $\uf,\vf \in \mathcal{H},$ and $\tau \geq0$.
Next, we set
\begin{align*}
\tilde{A}(\tau,\rho,x,y)=A\bigg(T-(T-T_0)e^{-\tau},(T-T_0) \rho e^{-\tau},
&(T-T_0)^{-\frac{2}{p-1}}e^{\frac{2}{p-1}\tau}\left(x+ i y\right)\bigg)\\
\tilde{B}(\tau,\rho,x,y)=B\bigg(T-(T-T_0)e^{-\tau},(T-T_0) \rho e^{-\tau},
&(T-T_0)^{-\frac{2}{p-1}}e^{\frac{2}{p-1}\tau}\left(x+ iy\right)\bigg)\\
\tilde{C}(\tau,\rho,x,y)=C\bigg(T-(T-T_0)e^{-\tau},(T-T_0) \rho e^{-\tau},
&(T-T_0)^{-\frac{2}{p-1}}e^{\frac{2}{p-1}\tau}\left(x+ i y\right)\bigg),
\end{align*} 
for $x,y\in \R$, to obtain
\begin{align*}
 &\left\|[\Vf(\uf,\tau)-\Vf(\vf,\tau)]_2\right\|_{\dot{H^1}(\B^3_1)}+\left\|[\Vf(\uf,\tau)-\Vf(\vf,\tau)]_4\right\|_{\dot{H^1}(\B^3_1)}\\
\leq&
(T-T_0)^{2+\frac{2}{p-1}}e^{-(2+\frac{2}{p-1})\tau}\Big\|\tilde{A}\left(\tau,.,\Re(u_1+iu_3),\Im(u_1+iu_3)\right)
\\
&-\tilde{A}\left(\tau,.,\Re(v_1+iv_3),\Im(v_1+iv_3)\right)\Big\|_{\dot{H^1}(\B^3_1)}
\\
&+(T-T_0)e^{-\tau}\sum_{l=1}^4\Big\|\tilde{B}\left(\tau,.,\Re(u_1+iu_3),\Im(u_1+iu_3)\right)u_l
\\
&-\tilde{B}\left(\tau,.,\Re(v_1+iv_3),\Im(v_1+iv_3)\right)v_l\Big\|_{\dot{H^1}(\B^3_1)}
\\
&+(T-T_0)e^{-\tau}\Big\|\tilde{C}\left(\tau,.,\Re(u_1+iu_3),\Im(u_1+iu_3)\right)(u_1'+iu_3')
\\
&-\tilde{C}\left(\tau,.,\Re(v_1+iv_3),\Im(v_1+iv_3)\right)(v_1'+i v_3')\Big\|_{\dot{H^1}(\B^3_1)}
\\
=&:I_1+I_2+I_3.
\end{align*}
A straightforward calculation then shows
\begin{align*}
I_1\lesssim &  (T-T_0)^{2+\frac{2}{p-1}}e^{-(2+\frac{2}{p-1})\tau}\Big\|\partial_2\tilde{A}\left(\tau,.,\Re(u_1+iu_3),\Im(u_1+iu_3)\right)
\\
&-\partial_2\tilde{A}\left(\tau,.,\Re(v_1+iv_3),\Im(v_1+iv_3)\right)\Big\|_{L^2(\B^3_1)}
\\
&+(T-T_0)^2e^{-2\tau}\bigg(\left\|\partial_3\tilde{A}\left(\tau,.,\Re(v_1+iv_3),\Im(v_1+iv_3)\right)\left(u_1'+iu_3'-(v_1'+iv_3')\right)\right\|_{L^2(\B^3_1)}
\\
&+\Big\|(u_1'+iu_3')\Big(\partial_3\tilde{A}\left(\tau,.,\Re(u_1+iu_3),\Im(u_1+iu_3)\right)
\\
&-\partial_3\tilde{A}\left(\tau,.,\Re(v_1+iv_3),\Im(v_1+iv_3)\right)\Big)\Big\|_{L^2(\B^3_1)}
\\
&+\left\|\partial_4\tilde{A}\left(\tau,.,\Re(v_1+iv_3),\Im(v_1+iv_3)\right)\left((u_1'+iu_3')-(v_1'+iv_3')\right)\right\|_{L^2(\B^3_1)}
\\
&+\Big\|(u_1'+iu_3')\Big(\partial_4\tilde{A}\left(\tau,.\Re(u_1+iu_3),\Re(u_1+iu_3)\right)
\\
&-\partial_4\tilde{A}\left(\tau,.,\Re(v_1+iv_3),\Im(v_1+iv_3)\right)\Big)\Big\|_{L^2(\B^3_1)}
\bigg)
\\
\lesssim& (T-T_0)^{\tilde{q}}e^{-\tilde{q}\tau}\Big(\left\|(u_1+iu_3)|u_1+i u_3|^{q-1}-(v_1+iv_3)|v_1+i v_3|^{q-1}\right\|_{L^2(\B^3_1)}\quad\quad\quad\quad\quad\quad\;\;
\\
&+\left\|v_1+i v_3 |v_1+i v_3|^{q-2})\left((u'_1+iu'_3)-(v_1'+ i v_3')\right)\right\|_{L^2(\B^3_1)}
\\
&+\left\|(u'_1+iu'_3)\left((u_1+i u_3) |u_1+i u_3|^{q-2}-(v_1+i v_3) |v_1+i v_3|^{q-2}\right)\right\|_{L^2(\B^3_1)}
\Big)
\\
\lesssim&
(T-T_0)^{\tilde{q}}e^{-\tilde{q}\tau}
\|\uf-\vf\|\Big(\|\uf\|^{q-1}+ \|\vf\|^{q-1}
\\
&+\|\uf\|\left(\left\||u_1+i u_3|^{q-2}\right\|_{L^\infty(\B^3_1)}+\left\||v_1+i v_3|^{q-2}\right\|_{L^\infty(\B^3_1)}\Big)
\right),
\end{align*}
for $\uf,\vf \in \mathcal{H}$ and $\tau \geq 0$.
Next, we again use Hölder's inequality and the Sobolev inequality $\|.\|_{L^4(\B^3_1)}\lesssim \|.\|_{H^1(\B^3_1)},$ in addition to the estimates on the perturbation, to obtain
\\
\begin{align*}
I_2
\lesssim &\sum_{l=1}^4\bigg[(T-T_0)e^{-\tau}\bigg(\left\| \partial_2\tilde{B}\left(\tau,.,\Re(v_1+iv_3),\Im(v_1+iv_3)\right)(u_l-v_l)\right\|_{L^2(\B^3_1)}
\\
&+\Big\|u_l\Big(\partial_2\tilde{B}\left(\tau,.,\Re(u_1+iu_3),\Im(u_1+iu_3)\right) 
\\
&-\partial_2\tilde{B}\left(\tau,.,\Re(v_1+iv_3),\Im(v_1+iv_3)\right)\Big)\Big\|_{L^2(\B^3_1)}
\bigg)
\\
&+(T-T_0)^{\tilde{q}}e^{-\tilde{q}\tau}\sum_{k=3,4}\bigg(\Big\|\partial_k\tilde{B}\left(\tau,.,\Re(u_1+iu_3),\Im(u_1+iu_3)\right)\\
&\times u_l(u_1'+iu_3'-(v_1'+iv_3'))\Big\|_{L^2(\B^3_1)}
\\
&+\left\|\partial_k\tilde{B}\left(\tau,.,\Re(u_1+iu_3),\Im(u_1+iu_3)\right)(v_1'+iv_3')(u_l-v_l) \right\|_{L^2(\B^3_1)}\bigg)
\\
&+(T-T_0)e^{-\tau}\Big\|\tilde{B}\left(\tau,.,\Re(v_1+iv_3),\Im(v_1+iv_3)\right) (u_l'-v_l')\Big\|_{L^2(\B^3_1)}
\\
&+(T-T_0)e^{-\tau}\Big\|u_l'\big(\tilde{B}\left(\tau,.,\Re(u_1+iu_3),\Im(u_1+iu_3)\right)
\\
&-\tilde{B}\left(\tau,.,\Re(v_1+iv_3),\Im(v_1+iv_3)\right)\big)\Big\|_{L^2(\B^3_1)}\bigg]
\\
\lesssim&(T-T_0)^{\tilde{q}}e^{-\tilde{q}\tau}\sum_{l=1}^4\Big(\|(1+|v_1+i v_3|)(u_l-v_l)\|_{L^2(\B^3_1)}+\|u_l(u_1+i u_3-(v_1+ iv_3))\|_{L^2(\B^3_1)}
\\
&+\left\|u_l\left(u_1'+iu'_3-(v_1'+i v_3')\right)\right\|_{L^2(\B^3_1)}
+\left\|(v_1'+i v_3')\left(u_l-v_l\right)\right\|_{L^2(\B^3_1)}
\\
&+ \|(1+|v_1+ iv_3|)(u_l'-v_l')\|_{L^2(\B^3_1)}+\|u_l'(u_1+i u_3-(v_1+ iv_3))\|_{L^2(\B^3_1)}\Big)
\\
\lesssim& (T-T_0)^{\tilde{q}}e^{-\tilde{q}\tau}\bigg(\|\uf-\vf\|(1+\|\uf\|+\|\vf\|)
\\
&+\sum_{l=1}^4\Big(\|u_l\|_{L^4(\B^3_1)}\|(u_1'+iu'_3)-(v_1'+i v_3')\|_{L^4(\B^3_1)}
\\
&+\|v_1'+i v_3'\|_{L^4(\B^3_1)}\|u_l- v_l\|_{L^4(\B^3_1)}
+\left(1+\|u_1+iu_3\|_{H^2(\B^3_1)}\right)\|u'_l-v'_l\|_{L^2(\B^3_1)}
\\
&+\|v_l'\|_{L^2(\B^3_1)}\|(u_1+ i u_3)-(v_1+ i v_3)\|_{H^2(\B^3_1)}\Big)\bigg)
\\
\lesssim& (T-T_0)^{\tilde{q}}e^{-\tilde{q}\tau}\|\uf-\vf\|\big(
1+\|\uf\|+\|\vf\|\big),
\end{align*}
for $\uf,\vf \in \mathcal{H}$ and $\tau \geq 0$.
Since one can obtain the estimate 
\begin{align*}
I_3\lesssim(T-T_0)^{\tilde{q}}e^{-\tilde{q}\tau}\|\uf-\vf\|\big(
1+\|\uf\|+\|\vf\|\big)
\end{align*}
analogously, the proof of this Lemma is finished.
\end{proof}
These two Lemmas will be vital for the fixed point argument which will be done later on. We continue by employing Duhamel's Principle, to rewrite Eq.~\eqref{eqtosovle} as an integral equation, which, for any initial data $\Phi(0)=\uf \in \mathcal{H}$, takes the form
\begin{align} \label{duhameleq}
\Phi(\tau)=& \Sf_{\theta_\infty}(\tau)\uf+\int_{0}^{\tau}\Sf_{\theta_\infty}(\tau-\sigma)\bigg(\hat{\Lf}_{\theta(\sigma)}\Phi(\sigma)+\N_{\theta(\sigma)}(\Phi(\sigma))\nonumber\\&+ \Vf(\Phi(\sigma)+\Psi_{\theta(\sigma)},\sigma)-\partial_\sigma \Psi_{\theta(\sigma)}\bigg) \,d\sigma
\end{align}
for $\tau \geq 0$ and with the abbreviation $\hat{\Lf}_{\theta(\sigma)}=\Lf'_{\theta(\sigma)}-\Lf'_{\theta_\infty}$. To analyse this equation further, we need the correct functional analytic setting and therefore introduce the two Banach spaces
$(\X,\|.\|_{\X})$ and $(X, \|.\|_X)$ as follows.

\begin{defi}
Set $\mathcal{X}:=\{\Phi \in C([0,\infty),\mathcal{H}):\|\Phi\|_{\X} < \infty \}$, with
$$\|\Phi\|_{\X}:=\sup_{\tau \geq 0}[e^{\omega_p \tau}\|\Phi(\tau)\|]
,$$
with $\omega_p:=\min\{\frac{\tilde{q}}{2},\frac{1}{p-1}\},$ where $\tilde{q}$ is the constant from Lemma \ref{lemonV}.
\end{defi}
The second Banach space that will be needed is the following. 
\begin{defi}
Let $X:=\{\theta \in C^1([0,\infty),\R):\theta(0)=0, \|\theta\|_X < \infty\},
$
where $$
\|\theta\|_X:=\sup_{\tau\geq 0 }\left[e^{\omega_p \tau}|\dot{\theta}(\tau)|+|\theta(\tau)|\right]. $$
\end{defi}
By $\X_\delta $ and $ X_\delta$ we denote the closed balls of radius $\delta$ in the corresponding norms. \\
Now follow two more lemmas that provide useful estimates.
\begin{lem}\label{manybounds}
Let $\Phi \in \X_\delta$ and $ \theta \in X_\delta$ with $0\leq\delta\leq \delta_0$ and $\delta_0$ sufficiently small. Further let $T_0 \in [1-\frac{3\delta}{c},1-\frac{2\delta}{c}]$ and $T \in [1-\frac{\delta}{c},1+\frac{\delta}{c}]$ for $c\geq 1$ sufficiently large. Then, we have the estimates
\begin{align*}
\|\hat{\Lf}_{\theta(\tau)}\Phi(\tau)\|+\|\N_{\theta(\tau)}(\Phi(\tau))\|&\lesssim \delta^2 e^{-2\omega_p\tau}\\
\|(\I- \Pf_{\theta_\infty} )\partial_\tau \Psi_{\theta(\tau)}\|+\|\Qf_{\theta_\infty}\partial_\tau \Psi_{\theta(\tau)}\|&\lesssim \delta^2 e^{-2\omega_p\tau}\\
\|\Vf(\Psi_{\theta(\tau)}+\Phi(\tau),\tau)\|&\lesssim \delta^2 e^{-2\omega_p\tau}
\end{align*} 
for all $\tau \geq 0$ and $\delta \in [0,\delta_0]$.
\end{lem}
\begin{proof}
By assumption $\theta$ is at least once continuously differentiable and therefore
\begin{align*}
|\theta(\tau_1)-\theta(\tau_2)|&\leq \int_{\tau_1}^{\tau_2}|\dot{\theta}(\sigma)|\, d\sigma\\
&\lesssim \delta\left( e^{-\omega_p\tau_1}+e^{-\omega_p\tau_2}\right).
\end{align*}
Since the expression $\delta\left( e^{-\omega_p\tau_1}+e^{-\omega_p\tau_2}\right)$ tends to $0$ as $\tau_1,\tau_2 \rightarrow \infty$, the limit $\theta_\infty:=\lim_{\tau\rightarrow \infty} \theta(\tau)$ exists and we even have the estimate
\begin{align*}
|\theta_\infty-\theta(\tau)|\leq \int_{\tau}^\infty |\dot{\theta}(\sigma)|\, d \sigma \lesssim \delta e^{-\omega_p \tau}.
\end{align*}
Hence Lemma \ref{liplem1} yields
\begin{align*}
\|\hat{\Lf}_{\theta(\tau)}\Phi(\tau)\|\leq \|\Lf_{\theta(\tau)}'-\Lf_{\theta_\infty}\|\, \|\Phi(\tau)\|\lesssim \delta e^{-\omega_p \tau}|\theta(\tau)-\theta_\infty|\lesssim \delta^2e^{-2\omega_p\tau}.
\end{align*}
Further, as $\N$ satisfies the quadratic estimate proven before, we have
$$\|\N_{\theta(\tau)}(\Psi(\tau))\|\lesssim \delta^2e^{-2\omega_p \tau}
$$
which establishes the first estimate. For the second one, note that 
\begin{align*}
\partial_\tau \Psi_{\theta(\tau)}=\dot{\theta}(\tau)\begin{pmatrix}
&-\Psi_{\theta,3}\,\,\,\,\\
&-\Psi_{\theta,4}\,\,\,\,\\
&\,\,\Psi_{\theta,1}\\
&\,\,\Psi_{\theta,2}
\end{pmatrix}
\end{align*}
which equals $\dot{\theta}(\tau)\rf_{\theta(\tau)}$. Set $\hat{\rf}_\theta:= \rf_{\theta(\tau)}-\rf_{\theta_\infty}$, to obtain
\begin{align*}
\|(\I-\Pf_{\theta_{\infty}}) \partial_\tau\Psi_{\theta(\tau)}\|&\lesssim\|(\I-\Pf_{\theta{_\infty}})\dot{\theta}(\tau) \rf_{\theta_\infty}\|+\|(\I-\Pf_{\theta_{\infty}})\dot{\theta}(\tau)\hat{\rf}_{\theta(\tau)}\|\\
&\lesssim |\dot{\theta}(\tau)|\,\|\hat{\rf}_{\theta(\tau)}\|\\
&\lesssim \delta e^{-\omega_p \tau}|\theta(\tau)-\theta_\infty|\\
&\lesssim \delta^2 e^{-2\omega_p \tau}.
\end{align*}
The estimate on $\Qf_{\theta_\infty}\partial_\tau \Psi_{\theta(\tau)}$ now follows from the same calculations since $\Qf_{\theta_\infty}\rf_{\theta_\infty}=\textup{\textbf{0}}$.
To obtain the last estimate, note that
\begin{align*}
\|\Vf(\Psi_{\theta(\tau)}(\tau)+\Phi(\tau),\tau)\|&\lesssim(T-T_0)^{\tilde{q}} e^{-\tilde{q}\tau}\left(1+\|\Psi_{\theta(\tau)}(\tau)+\Phi(\tau)\|^{q}+\|\Psi_{\theta(\tau)}(\tau)+\Phi(\tau)\|^{2}
\right)\\
&\lesssim (T-T_0)^{\tilde{q}} e^{-2\omega_p\tau}\\
&\lesssim \delta^2 e^{-2\omega_p\tau},
\end{align*}
provided $c$ is chosen large enough.
\end{proof}
Next, we also derive corresponding Lipschitz bounds.
\begin{lem}\label{manylip}
Let $\delta_0>0$ be small enough and $T_0 \in [1-\frac{3\delta}{c},1-\frac{2\delta}{c}]$ as well as $T \in [1-\frac{\delta}{c},1+\frac{\delta}{c}]$ where $c\geq 1$. Then, provided $c$ is chosen large enough, we have the estimates
\begin{align*}
\|\hat{\Lf}_{\theta_1(\tau)}\Phi_1(\tau)-\hat{\Lf}_{\theta_2(\tau)}\Phi_2(\tau)\|&\lesssim\delta e^{-2\omega_p \tau}\left(\|\Phi_1-\Phi_2\|_{\mathcal{X}}+\|\theta_1-\theta_2\|_X\right)\\
\|\N_{\theta_1(\tau)}(\Phi_1(\tau))-\N_{\theta_2(\tau)}(\Phi_2(\tau))\|&\lesssim\delta e^{-2\omega_p \tau}\left(\|\Phi_1-\Phi_2\|_{\mathcal{X}}+\|\theta_1-\theta_2\|_X\right)\\
\|(\I-\Pro_{\theta_{1_\infty}})\partial_\tau \Psi_{\theta_1(\tau)}-(\I-\Pro_{\theta_{2_\infty}})\partial_\tau\Psi_{\theta_2(\tau)}\|&\lesssim\delta e^{-2\omega_p\tau}\|\theta_1-\theta_2\|_{X}\\
\|\Qf_{\theta_{1_\infty}}\partial_\tau \Psi_{\theta_1(\tau)}-\Qf_{\theta_{2_\infty}}\partial_\tau \Psi_{\theta_2(\tau)}\|&\lesssim \delta e^{-2\omega_p \tau}\|\theta_1-\theta_2\|_X\\
\text{and}\\
\|\Vf(\Phi(\tau)_1+\Psi_{\theta_1(\tau)},\tau)-\Vf(\Phi_2(\tau)+\Psi_{\theta_2(\tau)},\tau)\|&\lesssim \delta e^{-2\omega_p \tau}\left(\|\Phi_1-\Phi_2\|_{\mathcal{X}} + \|\theta_1-\theta_2\|_X\right),
\end{align*}
for any $\Phi_1, \Phi_2 \in \mathcal{X}_\delta, \theta_1, \theta_2 \in X_\delta$, $\tau \geq 0$, and $\delta \in [0,\delta_0]$.
\end{lem}
\begin{proof}
Recall that 
\begin{align*}
\left[\Lf'_{\theta(\tau)} \uf\right]_2(\rho)=& c_p(p-1)\left(\cos(\theta(\tau))^2u_1(\rho)+\cos(\theta(\tau))\sin(\theta(\tau))u_3(\rho)\,\right)\\
&+c_p u_1(\rho).
\end{align*}
Therefore 
\begin{align*}
\left(\left[\Lf'_{\theta_\infty} \uf\right]_2-\left[\Lf'_{\theta(\tau)} \uf\right]_2\right)(\rho)=&
\int_\tau^\infty \dot{\theta}(\sigma)\bigg{(}c_p(p-1)\big(-2\cos(\theta(\sigma))\sin(\theta(\sigma))u_1(\rho)\\
&+\left(\cos(\theta)^2-\sin(\theta)^2\right)u_3(\rho)\,\big) 
\bigg{)}d\sigma.
\end{align*}
Thus, by setting $l_1(\theta):=-2\cos(\theta)\sin(\theta)$ and $l_2(\theta):=\cos(\theta)^2-\sin(\theta)^2$, we obtain
\begin{align*}
\left\|\left[\left(\hat{\Lf}_{\theta_1(\tau)}-\hat{\Lf}_{\theta_2(\tau)}\right)\uf\right]_2\right\|_{H^1(\B^3_1)}&\lesssim\bigg{(}\int_\tau^\infty |\dot{\theta}_1(\sigma)l_1(\theta_1(\sigma))-\dot{\theta}_2(\sigma)l_1(\theta_2(\sigma))|d\sigma\\
&+\int_\tau^\infty |\dot{\theta}_1(\sigma)l_2(\theta_1(\sigma))-\dot{\theta}_2(\sigma)l_2(\theta_2(\sigma))| d\sigma \bigg{)}\|\uf\|\\
&\lesssim \sum_{j=1}^2\bigg{(}\int_\tau^\infty|\dot{\theta}_1(\sigma)\left(l_j(\theta_1(\sigma))-l_j(\theta_2(\sigma))\right)| d \sigma\\
&+\int_\tau^\infty |(\dot{\theta}_1(\sigma)-\dot{\theta}_2(\sigma))l_j(\theta_2(\sigma))| d\sigma\bigg{)}\|\uf\|\\
&\lesssim \bigg{(}\int_\tau^\infty |\dot{\theta}_1(\sigma)|\, |\theta_1(\sigma)-\theta_2(\sigma)|\, +|\dot{\theta}_1(\sigma)-\dot{\theta}_2(\sigma)| d \sigma\bigg{)}\|\uf\|\\
&\lesssim \|\uf\| \int_\tau^\infty e^{-\omega_p \tau} \|\theta_1-\theta_2\|_{X} d\sigma\\
&\lesssim \|\uf\|e^{-\omega_p \tau} \|\theta_1-\theta_2\|_{X}.
\end{align*}
Analogously, one derives the bound $$\left\|\left[\left(\hat{\Lf}_{\theta_1(\tau)}-\hat{\Lf}_{\theta_2(\tau)}\right)\uf\right]_4\right\|_{H^1(\B^3_1)}\lesssim
 \|\uf\| e^{-\omega_p \tau} \|\theta_1-\theta_2\|_{X}.
$$
Hence 
\begin{align*}
\|\hat{\Lf}_{\theta_1(\tau)}\Phi_1(\tau)-\hat{\Lf}_{\theta_2(\tau)}\Phi_2(\tau)\|&\lesssim \|\hat{\Lf}_{\theta_1(\tau)}-\hat{\Lf}_{\theta_2(\tau)}\|\, \|\Phi_1(\tau)\|+\|\hat{\Lf}_{\theta_2(\tau)}\|\,\|\Phi_1(\tau)-\Phi_2(\tau)\|\\
&\lesssim \delta e^{-2\omega_p \tau}\left(\|\theta_1-\theta_2\|_{X}+\|\Phi_1-\Phi_2\|_{\mathcal{X}}\right).
\end{align*}
The estimate on the nonlinearity follows immediately from Lemma \ref{nonesti}.
To derive the third bound stated in the Lemma, recall that
 $(\I- \Pro_{\theta_\infty}) \partial_\tau \Psi_{\theta(\tau)}=\dot{\theta}(\tau)(\I-\Pro_{\theta_\infty}) \hat{\rf}_{\theta(\tau)}$ for any $\theta \in \R$. Furthermore, since the function $(\theta,\rho) \mapsto \rf_\theta(\rho)$
is smooth for any $\theta$, the representation
$$
\hat{\rf}_{\theta(\tau)}=-\int_\tau^\infty \partial_\sigma\rf_{\theta(\sigma)}(\rho) \,d\sigma
$$
implies that
\begin{align*}
\|\hat{\rf}_{\theta_1(\tau)}-\hat{\rf}_{\theta_2(\tau)}\|&\leq \int_\tau^\infty\|\partial_\sigma (\rf_{\theta_1(\sigma)}-\rf_{\theta_2(\sigma)})\| \,d\sigma\\
&\lesssim \int_\tau^\infty e^{-\omega_p\sigma}\|\theta_1-\theta_2\|_X \,d\sigma\\
&\lesssim e^{-\omega_p\tau}\|\theta_1-\theta_2\|_X .
\end{align*}

Thus we obtain
\begin{align*}
\|(\I-\Pro_{\theta_{1_\infty}})\partial_\tau \Psi_{\theta_1(\tau)}-(\I-\Pro_{\theta_{2_\infty}})\partial_\tau\Psi_{\theta_2(\tau)}\|\lesssim & \|\dot{\theta}_1(\tau)(\I-\Pro_{\theta_{1_\infty}}) \hat{\rf}_{\theta_1(\tau)}-\dot{\theta}_2(\tau) (\I-\Pro_{\theta_{2_{\infty}}}) \hat{\rf}_{\theta_2(\tau)}\|\\
\lesssim & 
\|\dot{\theta}_1(\tau)(\I-\Pro_{\theta_{1_\infty}})-\dot{\theta}_2(\tau)(\I-\Pro_{\theta_{2_\infty}})\|\,\|\hat{\rf}_{\theta_1(\tau)}\|\\
&+\|\dot{\theta}_2(\tau)(\I-\Pro_{\theta_{2_\infty}})\|\,\|\hat{\rf}_{\theta_1(\tau)}-\hat{\rf}_{\theta_2(\tau)}\|\\
\lesssim & \delta e^{-2\omega_p \tau}\|\theta_1-\theta_2\|_X.
\end{align*}
To prove the fourth estimate, note that we again have 
$$\Qr_{\theta_\infty} \partial_\tau \Psi_{\theta(\tau)}=\dot{\theta}(\tau)\Qr_{\theta_\infty}\hat{\rf}_{\theta(\tau)}$$  and the same considerations done to establish the third claim also yield the fourth one.
To establish the final inequality, note that for $\varepsilon$ small enough the Sobolev embedding $H^2(\B^3_1)\hookrightarrow L^\infty(\B^3_1)$ implies that $|u(\rho)|\leq\frac{1}{4}$, for all $u\in H^2(\B^3_1)$ with $\|u\|_{ H^2(\B^3_1)}\leq \varepsilon.$ This in turn implies that there exists a $k>0$ such at
$$ 
\left|[\Psi_{\theta(\tau)}]_1+i[\Psi_{\theta(\tau)}]_3+u_1+iu_3\right|\geq k,
$$
for all $\theta \in X_\delta$ and all $\uf \in \mathcal{H}$ with $\|\uf\|_{\mathcal{H}}\leq \delta,$ for $\delta <\varepsilon$.
Thus if $\delta$ is chosen small enough we have the estimate
\begin{align*}
	\|\Vf(\Phi_1(\tau)+\Psi_{\theta_1(\tau)},\tau)-\Vf(\Phi_2(\tau)+\Psi_{\theta_2(\tau)},\tau)\|\lesssim& (T-T_0)^{\tilde{q}}e^{-\tilde{q}\tau}\bigg( \|\Phi_1(\tau)-\Phi_2(\tau)\|\\
&+\|\Psi_{\theta_1(\tau)}-\Psi_{\theta_2(\tau)}\|
\bigg)\\
\lesssim& (T-T_0)^{\tilde{q}}e^{-2\omega_p\tau}\big(\|\Phi_1-\Phi_2\|_{\mathcal{X}}\\ &+\|\theta_1-\theta_2\|_X \big),
\end{align*}
due to Lemma \ref{lemonV}. Thus, the claim follows if $c$ is chosen large enough.
\end{proof}
Beginning with the next section, we always assume that $c,$ $T_0$ and $T$ are chosen such that the Lemmas \ref{manybounds} and \ref{manylip} hold, without stating it explicitly.
\section{Unstable subspaces}
Our next step is to deal with the unstable subspaces $\rg \,\Pro_{\theta_\infty}$ and $\rg \, \Qr_{\theta_\infty}$ which are induced by the invariances of our equation.
\subsection{The modulation equation}
The instability corresponding to the eigenvalue $0$ will be handled by modulation, i.e., by finding a function $\theta(\tau)$ such that the instability is completely suppressed. To derive an equation for such a $\theta$, we formally apply the projection $\Pf_{\theta_\infty}$ to Eq.~\ref{duhameleq}. This then yields
\begin{align}\label{goal}
\Pf_{\theta_\infty} \Phi(\tau) =&\Pf_{\theta_\infty} \uf \nonumber
\\
&+\Pf_{\theta_\infty}  \int_0^\tau \left(\hat{\Lf}_{\theta(\sigma)}\Phi(\sigma) +\N_{\theta(\sigma)}(\Phi(\sigma)+\Vf(\Phi(\sigma)+\Psi_{\theta(\sigma)}, \sigma)-\partial_\sigma \Psi_{\theta(\sigma)}\right)d\sigma.
\end{align}
The idea now is to set the right-hand side equal to zero. But as this would entail the boundary condition $\Pf_{\theta_\infty} \uf=\textup{\textbf{0}}$ for $\tau =0$, which is not always satisfied, we have to use a small trick. To this end, denote by $\chi:[0,\infty)\to [0,1]$ a smooth cut-off function that satisfies $\chi(\tau)=1$ for $\tau \in [0,1],\, \chi(\tau)=0$ for $ \tau \geq 4$ and finally $ |\chi'(\tau)| \leq 1$ for all $\tau \geq 0$. Next, we make the ansatz $\Pf_{\theta_\infty} \Phi(\tau)= \chi(\tau)\tilde{\rf}$ for some $\tilde{\rf} \in \rg \Pf_{\theta_\infty}$.
Since evaluation at the time $ \tau=0$ implies $\tilde{\rf}=\Pf_{\theta_\infty}\uf$, one obtains the modulation equation
\begin{align}\label{modeq}
\left(1-\chi(\tau)\right) &\Pf_{\theta_\infty} \uf\nonumber
\\+&\Pf_{\theta_\infty}  \int_0^\tau \left(\hat{\Lf}_{\theta(\sigma)}\Phi(\sigma) +\N_{\theta(\sigma)}(\Phi(\sigma)+\Vf(\Phi(\sigma)+\Psi_{\theta(\sigma)}, \sigma)-\partial_\sigma \Psi_{\theta(\sigma)}\right) \, d\sigma=0.
\end{align}
Now note that making the further assumption $\theta(0)=0$ yields
\begin{align*}
\Pf_{\theta_\infty}\int_0^{\tau} \partial_\sigma \Psi_{\theta(\sigma)} \, d\sigma &=\Pf_{\theta_\infty}\int_0^\tau \dot{\theta}(\sigma)\left(\rf_{\theta_\infty}+\hat{\rf}_{\theta(\sigma)}\right)\, d \sigma \\
&=\left(\theta(\tau) \rf_{\theta_\infty}+\Pf_{\theta_\infty}\int_0^\tau \dot{\theta}(\sigma)\hat{\rf}_{\theta(\sigma)} d \sigma\right).
\end{align*}
If we insert this into Eq.~\eqref{modeq}, we obtain
\begin{align}\label{modeq2}
\theta(\tau) \rf_{\theta_\infty}= \left(1-\chi(\tau)\right) \Pf_{\theta_\infty} \uf+&\Pf_{\theta_\infty}  \int_0^\tau \hat{\Lf}_{\theta(\sigma)}\Phi(\sigma) +\N_{\theta(\sigma)}\left(\Phi(\sigma)\right)+\Vf(\Phi(\sigma)+\Psi_{\theta(\sigma)}, \sigma) \, d\sigma
\\
-&\Pf_{\theta_\infty}\int_0^\tau \dot{\theta}(\sigma)\hat{\rf}_{\theta(\sigma)}\, d\sigma. \nonumber
\end{align}
The next Lemma will show that, provided $\Phi$ is sufficiently small in norm,  there is indeed a $\theta :[0,\infty) \rightarrow \R$ such that Eq.~\eqref{modeq2} is satisfied.

\begin{lem}\label{solutlem}
Suppose $\delta >0$ is sufficiently small and $c>1$ is sufficiently large. Furthermore let $\Phi \in \mathcal{X}_\delta$ and $\uf \in \mathcal{H}$ with $\|\uf\| \leq \frac{\delta}{c}$. Then there exists a unique function $\theta \in X_\delta$ such that $\theta$ satisfies Eq.~\eqref{modeq2} and such that the map
$\Phi \mapsto \theta: \mathcal{X}_\delta \subset \mathcal{X} \rightarrow X$ is Lipschitz-continuous.  
\end{lem}

\begin{proof}
The idea to prove this, is by setting up a contraction for $\theta$. To this end, we begin by rewriting Eq.~\eqref{modeq2} as
\begin{align*}
\theta(\tau)\rf_{\theta_\infty}=&-\int_0^{\tau}\chi'(\sigma) \Pf_{\theta_\infty}\uf\, d\sigma\\
&+\int_0^\tau \Pf_{\theta_\infty}\left(  \hat{\Lf}_{\theta(\sigma)}\Phi(\sigma) +\N_{\theta(\sigma)}(\Phi(\sigma))+\Vf(\Phi(\sigma)+\Psi_{\theta(\sigma)}, \sigma) \right)\, d\sigma
\\
&-\int_0^\tau \dot{\theta}(\sigma)\Pf_{\theta_\infty}\hat{\rf}_{\theta(\sigma)}\, d\sigma
\\
=&:\int_0^\tau \G(\theta,\Phi,\uf)(\sigma)\, d\sigma.
\end{align*}
This yields 
$$\theta(\tau)\|\rf_{\theta_\infty}\|^2=\left(\int_0^\tau \G(\theta,\Phi,\uf)(\sigma)\, d\sigma\Big|\rf_{\theta_\infty}\right).
$$
Therefore, by setting $$\tilde{G}(\theta,\Phi,\uf)(\sigma)= \|\rf_{\theta_\infty}\|^{-2}\left(\G(\theta,\Phi,\uf)(\sigma)\big|\rf_{\theta_\infty}\right),$$
we obtain
$$\theta(\tau)=\int_0^\tau\tilde{G}(\theta,\Phi,\uf)(\sigma)\,d\sigma =:G(\theta,\Phi,\uf)(\tau).
$$
Thanks to Lemma \ref{liponrg}, we know that 
$$
\|\hat{\rf}_{\theta(\tau)}\|\lesssim |\theta(\tau)-\theta_\infty|\lesssim \delta e^{-\omega_p\tau}.
$$
Further, it is also clear that
$$
\|\chi'(\tau) \Pf_{\theta_\infty}\uf\|\lesssim \|\chi'(\tau)\uf\|\lesssim \frac{\delta}{c}e^{-2\omega_p\tau}.
$$
Thus Lemma \ref{manybounds} implies $\| \G(\theta,\Phi,\uf)(\tau)\|\lesssim (\frac{\delta}{c}+\delta^2)e^{-2\omega_p \tau}$, from which we conclude that for $\theta\in X_\delta$ we have $$
G(\theta,\Phi,\uf)\in X_\delta$$ provided that $\delta >0$ and $ c>0$ are chosen sufficiently small and large, respectively.
Next, note that 
$$
\|\chi'(\tau) \Pf_{\theta_{1_\infty}}\uf-\chi'(\tau) \Pf_{\theta_{2_\infty}}\uf\| \lesssim \delta e^{-2\omega_p \tau}|\theta_{1_\infty}-\theta_{2_\infty}|\lesssim \delta e^{-2\omega_p \tau}\|\theta_1-\theta_2\|_X$$
as well as 
\begin{align*}
\|\dot{\theta}_1(\tau)\Pf_{\theta_{1_\infty}}\hat{\rf}_{\theta_1(\tau)}-\dot{\theta}_2(\tau)\Pf_{\theta_{2_\infty}}\hat{\rf}_{\theta_2(\tau)}\|\leq& \|\dot{\theta}_1(\tau)\Pf_{\theta_{1_\infty}}\hat{\rf}_{\theta_1(\tau)}-\dot{\theta}_1(\tau)\Pf_{\theta_{2_\infty}}\hat{\rf}_{\theta_2(\tau)}\|\\
+&\|\dot{\theta}_1(\tau)\Pf_{\theta_{2_\infty}}\hat{\rf}_{\theta_2(\tau)}-\dot{\theta}_2(\tau)\Pf_{\theta_{2_\infty}}\hat{\rf}_{\theta_2(\tau)}\|\\ \lesssim &\delta e^{-\omega_p \tau}\|\Pf_{\theta_{1_\infty}}\hat{\rf}_{\theta_1(\tau)}-\Pf_{\theta_{2_\infty}}\hat{\rf}_{\theta_2(\tau)}\|+\delta e^{-\omega_p\tau}|\dot{\theta}_1-\dot{\theta}_2|\\
\lesssim & \delta e^{-2\omega_p \tau}\|\theta_1-\theta_2\|_X,
\end{align*}
since $\|\hat{\rf}_{\theta_{1_\infty}}-\hat{\rf}_{\theta_{2_\infty}}\|\lesssim e^{-\omega_p \tau}\|\theta_1-\theta_2\|_X.$
These two estimates together with the estimates provided by Lemma \ref{manylip} now imply that
$$\|\G(\theta_1,\Phi,\uf)(\tau)-\G(\theta_2,\Phi,\uf)(\tau)\|\lesssim \delta e^{-2\omega_p \tau}\|\theta_1-\theta_2\|_X
$$
which in turn yields
\begin{equation}\label{notneededagain}
\|G(\theta_1,\Phi,\uf)-G(\theta_2,\Phi,\uf)\|_X\lesssim \delta\|\theta_1-\theta_2\|_X,\end{equation}
for all $\theta_1, \theta_2 \in X_\delta$. Therefore, the requirements of the contraction mapping principle are satisfied and we obtain the existence of a unique $\theta \in X_\delta$ with $\theta(\tau)=G(\theta,\Phi,\uf)(\tau)$. To prove the final claim, let $\theta_1(\tau)=G(\theta_1,\Phi_1,\uf)(\tau)$ and $\theta_2(\tau)= G(\theta_2,\Phi_2,\uf)(\tau),$
for $\Phi_1, \Phi_2 \in \mathcal{X}_\delta.$ We then estimate
\begin{align*}
|\dot{\theta}_1(\tau)-\dot{\theta}_2(\tau)|\lesssim& |\,G(\theta_1,\Phi_1,\uf)-G(\theta_2,\Phi_2,\uf)\,|\\
\lesssim& \delta e^{-2\omega_p \tau}\left(\|\Phi_1-\Phi_2\|_{\mathcal{X}}+\|\theta_1-\theta_2\|_X \right),
\end{align*}
due to the previous considerations in this proof and the estimates in Lemma \ref{manylip}. 
This now yields the claim by invoking the fundamental theorem of calculus, provided $\delta$ is chosen sufficiently small.
\end{proof}

\subsection{Time-translation instability}

Now we deal with the unstable subspace  $\rg\, \Qr_{\theta_\infty}.$ This will be done by adding a correction term to the evolution in order to stabilize it. To find such a term, we formally apply $ \Qr_{\theta_\infty}$ to Eq.~\eqref{modeq} which yields
\begin{align*}
\Qr_{\theta_\infty}\Phi(\tau)=&\,e^{\tau} \Qr_{\theta_\infty}\uf\\
&+e^{\tau}\Qr_{\theta_\infty}\int_0^\tau e^{-\sigma} \hat{\Lf}_{\theta(\sigma)}\Phi(\sigma) +\N_{\theta(\sigma)}(\Phi(\sigma))+\Vf(\Phi(\sigma)+\Psi_{\theta(\sigma)}, \sigma)\, d\sigma\\
&-e^{\tau}\Qr_{\theta_\infty}\int_0^{\tau}e^{-\sigma}\partial_\sigma \Psi_{\theta(\sigma)} \, d\sigma.
\end{align*}
Therefore, we set \begin{align*}
\Cf(\Phi,\theta,\uf):=\,& \Qr_{\theta_\infty} \uf\\
+&\Qr_{\theta_\infty}\int_0^\infty e^{-\sigma}\left( \hat{\Lf}_{\theta(\sigma)}\Phi(\sigma) +\N_{\theta(\sigma)}(\Phi(\sigma))+\Vf(\Phi(\sigma)+\Psi_{\theta(\sigma)}, \sigma)-\partial_\sigma \Phi_{\theta(\sigma)}\right) \, d\sigma
\end{align*}
and first deal with the modified equation given by 
\begin{align}\label{eqmodi}
\Phi(\tau)=&\Sf_{\theta_\infty}(\tau)\left(\uf-\Cf(\Phi,\theta,\uf)\right)\nonumber\\
&+
\int_0^{\tau}\Sf_{\theta_\infty}(\tau-\sigma)\left(\hat{\Lf}_{\theta(\sigma)}\Phi(\sigma) +\N_{\theta(\sigma)}(\Phi(\sigma))+\Vf(\Phi(\sigma)+\Psi_{\theta(\sigma)}, \sigma)-\partial_\sigma \Phi_{\theta(\sigma)}\right) \, d\sigma.
\end{align}
\begin{prop}
Let $\delta>0 $ be small enough and $c\geq 1$ be sufficiently large. For any $\uf \in \mathcal{H}$ with $\|\uf\|\leq \frac{\delta}{c}$ there exist unique functions $\Phi \in \mathcal{X}_\delta$ and $\theta \in X_\delta$ such that equation \eqref{eqmodi} holds for all $\tau \geq 0$.
\end{prop}
\begin{proof}
To begin with, we denote the right-hand side of Eq.~\eqref{eqmodi} by $\K(\Phi,\theta,\uf)(\tau).$ The idea of this proof is to again invoke the contraction mapping principle. Therefore, we first claim that for $\delta >0$ small enough and $\Phi \in \mathcal{X}_\delta$ we have $ \K(\Phi,\theta,\uf) \in \mathcal{X}_\delta$ where $\theta \in X_\delta$ is the one associated to $\Phi$ by Lemma \ref{solutlem}.
We first apply the projection $\Qr_{\theta_\infty}$ to the right-hand side of the equation which yields
\begin{align*}
\Qf_{\theta_\infty} \K(\Phi,\theta,\uf)(\tau)=&-\int_\tau^\infty
e^{\tau-\sigma} \Qf_{\theta_\infty}\bigg{(}\hat{\Lf}_{\theta(\sigma)}\Phi(\sigma) +\N_{\theta(\sigma)}(\Phi(\sigma)\\
&+\Vf(\Phi(\sigma)+\Psi_{\theta(\sigma)}, \sigma)-\partial_\sigma \Phi_{\theta(\sigma)}\bigg{)} \, d\sigma.
\end{align*}
Therefore, Lemma \ref{manybounds} implies
\begin{align*}
\|\Qf_{\theta_\infty} \K(\Phi,\theta,\uf)(\tau)\|\lesssim \delta^2 \int_{\tau}^\infty e^{\tau-\sigma}e^{-2\omega_p\tau}\lesssim \delta^2 e^{-2\omega_p \tau}.
\end{align*}
Thus $\Qf_{\theta_\infty} \K(\Phi,\theta,\uf) \in \mathcal{X}_{\frac{\delta}{4}}$ if $\delta$ is chosen small enough. Note that since $\Qr_{\theta_\infty}\Pro_{\theta_\infty}=0$ and $\Cf(\Phi,\theta,\uf)$ is contained in the range of $\Qf_{\theta_\infty}$, we obtain
\begin{align*}
\Pf_{\theta_\infty}\K(\Phi,\theta,\uf)(\tau)=\chi(\tau)\Pf_{\theta_\infty}\uf,
\end{align*}
due to Eq.~\eqref{modeq}.
Hence
\begin{align*}
\|\Pf_{\theta_\infty}\K(\Phi,\theta,\uf)(\tau)\|\lesssim \frac{\delta}{c}e^{-2\omega_p\tau}
\end{align*}
and, provided $c$ is chosen sufficiently large, this implies
$\Pf_{\theta_\infty}\K(\Phi,\theta,\uf) \in \mathcal{X}_{\frac{\delta}{4}}$.
To show that $\K(\Phi,\theta,\uf) \in \mathcal{X}_\delta$, it remains to consider $\left(\I-\Pro_{\theta_\infty}-\Qr_{\theta_\infty}\right)\K(\Phi,\theta,\uf)(\tau),$ which equals
\begin{align*}
&\Sf_{\theta_\infty}(\tau)\tilde{\Pro}_{\theta_\infty}\left(\uf+\Cf(\Phi,\theta,\uf)\right)\\
+\int_0^\tau&\Sf_{\theta_\infty}(\tau-\sigma) \tilde{\Pro}_{\theta_\infty} \left(\hat{\Lf}_{\theta(\sigma)}\Phi(\sigma) +\N_{\theta(\sigma)}(\Phi(\sigma))
+\Vf(\Phi(\sigma)+\Psi_{\theta(\sigma)}, \sigma)-\partial_\sigma \Phi_{\theta(\sigma)}\right) \, d\sigma
\end{align*}
with $\tilde{\Pro}_{\theta_\infty}$ defined as in Proposition \ref{projections}.
Thanks to Lemma \ref{manylip}, we have
$$
\|\Cf(\Phi,\theta,\uf)\|\lesssim \frac{\delta}{c}+\delta^2\int_0^\infty e^{-\sigma-2\omega_p\sigma} \, d\sigma\lesssim \frac{\delta}{c}+\delta^2
$$
and therefore, from Proposition \ref{projections} and Lemma \ref{manybounds}, we infer
\begin{align*}
\|\left(\I-\Pro_{\theta_\infty}-\Qr_{\theta_\infty} \right)\K(\Phi,\theta,\uf)(\tau)\|
&\lesssim (\frac{\delta}{c}+\delta^2)e^{-\omega_p\tau}+\delta^2\int_0^{\tau}e^{-\omega_p(\tau-\sigma)}e^{-2\omega_p\sigma}\, d\sigma \\ 
&\lesssim (\frac{\delta}{c}+\delta^2)e^{-\omega_p\tau}.
\end{align*}
In summary, $\K(\Phi,\theta,\uf)\in \mathcal{X}_\delta$ whenever $\Phi \in \mathcal{X}_\delta$.
It remains to show the Lipschitz-continuity of $\K(\Phi,\theta,\uf)$ for arbitrary $\Phi \in \mathcal{X}_\delta$.
Hence let $\Phi_1, \Phi_2 \in \mathcal{X}_\delta$ and let $\theta_i$ be associated to $\Phi_i$ through Lemma \ref{solutlem}. We proceed in a similar manner as before and therefore first deal with $\Qf_{\theta_\infty}\K(\Phi,\theta,\uf)$. This yields
\begin{align*}
\|&\Qf_{\theta_{1_\infty}}\K(\Phi_1,\theta_1,\uf)(\tau)-\Qf_{\theta_{2_\infty}}\K(\Phi_2,\theta_2,\uf)(\tau)\|
\\
\lesssim& \|\Qf_{\theta_{1_\infty}}\K(\Phi_1,\theta_1,\uf)(\tau)-\Qf_{\theta_{1_\infty}}\K(\Phi_2,\theta_2,\uf)(\tau\textbf{})\| \\
&+\|\Qf_{\theta_{1_\infty}}\K(\Phi_2,\theta_2,\uf)(\tau)-\Qf_{\theta_{2_\infty}}\K(\Phi_2,\theta_2,\uf)(\tau)\| \\
\lesssim & \delta \int_\tau^{\infty} e^{\tau-\sigma}e^{-2\omega_p\sigma}\|\Phi_1-\Phi_2\|_{\mathcal{X}} d \sigma
\\
\lesssim & \delta e^{-2\omega_p\tau}\|\Phi_1-\Phi_2\|_{\mathcal{X}},
\end{align*}
by using the estimates given in Lemma \ref{manylip} as well as the estimate $$\|\theta_1-\theta_2\|_X\lesssim \|\Phi_1-\Phi_2\|_{\mathcal{X}}$$ which was derived in Lemma \ref{solutlem}.
We further obtain
\begin{align*}
\|\Pf_{\theta_{1_\infty}}\K(\Phi_1,\theta_1,\uf)(\tau)-\Pf_{\theta_{2_\infty}}\K(\Phi_2,\theta_2,\uf)(\tau)\|&\lesssim |\chi(\tau)|\,\|(\Pf_{\theta_{1_\infty}}-\Pf_{\theta_{2_\infty}})\uf\|\\
&\lesssim \delta e^{-2\omega_p\tau}\|\theta_1-\theta_2\|_X\\
&\lesssim \delta e^{-2\omega_p \tau} \|\Phi_1-\Phi_2\|_\X,
\end{align*}
due to the Lemmas \ref{liponrg} and \ref{solutlem}.
By invoking Lemmas \ref{liponrg} and \ref{manylip}, one obtains \begin{align*}
\|\Cf(\Phi_1,\theta_1,\uf)-\Cf(\Phi_2,\theta_2,\uf)\|&\lesssim \delta\|\Phi_1-\Phi_2\|_\X+\delta \int_0^{\infty} e^{-\sigma-2\omega_p\sigma}\|\Phi_1-\Phi_2\|_{\mathcal{X}}\\
&\lesssim\delta\|\Phi_1-\Phi_2\|_{\mathcal{X}}.
\end{align*}
Hence,
\begin{align*}
\|\tilde{\Pro}_{\theta_{1_\infty}}\K(\Phi_1,\theta_1,\uf)(\tau)-\tilde{\Pro}_{\theta_{2_\infty}}\K(\Phi_2,\theta_2,\uf)(\tau)\|&\lesssim \delta e^{-\omega_p \tau}\|\Phi_1-\Phi_2\|_{\mathcal{X}}
\\
&+\delta\int_0^\tau e^{-\omega_p(\tau+\sigma)}\|\Phi_1-\Phi_2\|_{\mathcal{X}}\\
&\lesssim \delta e^{-\omega_p\tau}\|\Phi_1-\Phi_2\|_{\mathcal{X}}
\end{align*}
again with the help of the Lemmas \ref{liponrg} and \ref{manylip}.
The claim is now established, since the conditions of the aforementioned contraction mapping principle have been established.
\end{proof}
\subsection{Variation of blowup time}
We are now going to develop tools that allow us to solve Eq.~\eqref{eqmodi} without the correction term $\Cf(\Phi,\theta,\uf)$. In order to do so, we first introduce the scaling operator $\vf \mapsto \vf^T$ with
\begin{align*}
\vf^T(\rho):=\begin{pmatrix}
&(T-T_0)^{\frac{2}{p-1}}v_1((T-T_0) \rho)\\&(T-T_0) ^{\frac{p+1}{p-1}}v_2((T-T_0) \rho)\\
&(T-T_0) ^{\frac{2}{p-1}}v_3((T-T_0) \rho)\\&(T-T_0) ^{\frac{p+1}{p-1}}v_4((T-T_0) \rho)
\end{pmatrix}
\end{align*}
for any $\vf \in \mathcal{H}(\B^3_R):= \{\vf \in(H^2\times H^1(\B^3_R))^2:\vf \text{ radial} \}$.
We do this since, due to transformations we applied to Eq.~\eqref{startingeq}, the blowup time is now also showing up in the initial data of Eq.~\eqref{eq2}. By the splitting \ref{splitting}, the initial data can the be rewritten as
\begin{align*}
\Phi(0)=\Psi(0)-\Psi_{\theta(0)}(0)=
\Jf(\tilde{f},\tilde{g})^T+\Psi_0^T-\Psi_0
\end{align*}
where $\Jf:H^2\times H^1(\B^3_R)\rightarrow \mathcal{H}(\B^3_R)$ with

$$\Jf(\tilde{f},\tilde{g})(\rho):=\begin{pmatrix}
&\Re\left(\tilde{f}(\rho)\right)\\
&\Re\left(\tilde{g}(\rho)\right)\\
&\Im\left(\tilde{f}(\rho)\right)\\
&\Im\left(\tilde{g}(\rho)\right)
\end{pmatrix}
$$
Note that $\Jf$ is a bounded linear operator. Motivated by this, we set
\begin{align*}
\Uf(T,\vf):=\Jf(\vf)^T+\Psi_0^T-\Psi_0
\end{align*}
for any $\vf \in H^2\times H^1(\B^3_R).$ Further, by setting $ \vf:=(\tilde{f},\tilde{g})$ we can rewrite our initial data as
$$\Phi(0)=\Uf(T,\vf).
$$
Next, we need the following result on $\Uf$.
\begin{lem}\label{bigUlem}
Let $\delta >0$ be small enough, $ c\geq 1$ sufficiently large, $T_0 \in [1-\frac{3\delta}{c},1-\frac{2\delta}{c}]$ and $T \in[1-\frac{\delta}{c^2},1+\frac{\delta}{c^2}]$. Furthermore, suppose $\|\vf\|_{ H^2\times H^1(\B_R^3)}$ is sufficiently small,  where $R=\frac{3\delta}{c}+\frac{\delta}{c^2}.$ Then

$$\|\Uf(T,\vf)\|\leq \delta,$$
for all $ T \in[1-\frac{\delta}{c^2},1+\frac{\delta}{c^2}]$ and the map 
$T \mapsto \Uf(T,\vf):[1-\frac{\delta}{c^2},1+\frac{\delta}{c^2}] \rightarrow \mathcal{H}$ is continuous.
\end{lem}
\begin{proof}
First note that \begin{align*}
\left|[\Psi_0^T]_1-[\Psi_0]_1\right|&=\kappa_p\left|
\frac{(T-T_0)^{\frac{2}{p-1}}}{(1-T_0)^{\frac{2}{p-1}}} -1\right|\\&\lesssim 
\left|\frac{T-T_0-1+T_0}{1-T_0}\right|^{\frac{2}{p-1}}
\\
&\leq
\left|\frac{1-T}{1-T_0}\right|^{\frac{2}{p-1}}\\
&\lesssim
\left|c\right|^{-\frac{2}{p-1}}.
\end{align*}
Further, the same estimate holds true for the other components as well.
To continue let $\vf \in \mathcal{H}(\B_R^3)$. Then
\begin{align*}
\|\vf^T\|\lesssim&  (T-T_0)^{\frac{2}{p-1}-3}\|\vf\|_{H^2(\B_{T-T_0}^3)}\\
\lesssim& \left(\frac{c}{\delta}\right)^{3-\frac{2}{p-1}}\|\vf\|_{H^2(\B_R^3)}.
\end{align*} 
Therefore, as $\|\Jf\|=1$, we obtain $\|\Uf(T,\vf)\|\leq\delta$, provided $\vf$ satisfies 
$$\|\vf\|_{H^2\times H^1(\B_R^3)}\leq \delta^{3-\frac{2}{p-1}}c^{-3},$$
and $c$ is chosen large enough.

\end{proof}
This result immediately implies the following Lemma.
\begin{lem}\label{solutionU}
Assume that the assumptions of Lemma \ref{bigUlem} hold. Then the equation 
\begin{align}\label{eqbigU}
\Phi(\tau)= & \,\Sf_{\theta_\infty}(\tau)\left(\Uf(T,\vf)-\Cf(\Phi,\theta,\Uf(T,\vf))\right)\\
&+
\int_0^{\tau}\Sf_{\theta_\infty}(\tau-\sigma)\left(\hat{\Lf}_{\theta(\sigma)}\Phi(\sigma) +\N_{\theta(\sigma)}(\Phi(\sigma))+\Vf(\Phi(\sigma)+\Psi_{\theta(\sigma)}, \sigma)-\partial_\sigma \Phi_{\theta(\sigma)}\right) \, d\sigma\nonumber
\end{align}
has a solution $(\Phi,\theta) \in \mathcal{X}_\delta \times X_\delta$ and the solution map $T \mapsto (\Phi,\theta):[1-\frac{\delta}{c^2},1+\frac{\delta}{c^2}]\rightarrow\mathcal{X}\times X$ is continuous.
\end{lem}
\begin{proof}
This follows immediately from Lemma \ref{bigUlem}, since one obtains Eq.~\eqref{eqbigU} by replacing $\uf$ with $\Uf(T,\vf)$ in Eq.~\eqref{eqmodi}.
\end{proof}
The final Lemma, needed to prove the main result, states that we can choose a $T$ such that the correction term $\Cf(\Phi,\theta,\Uf(T,\vf))$ vanishes.

\begin{lem}\label{Cvanish}
Let the assumptions be as in Lemma \ref{bigUlem}. Then, there exist functions $(\Phi,\theta)\in \mathcal{X}_\delta \times X_\delta$ and a $T \in [1-\frac{\delta}{c^2},1+\frac{\delta}{c^2}]$ such that Eq.~\eqref{eqbigU} holds with $\Cf(\Phi,\theta,\Uf(T,\vf))=\textup{\textbf{0}}$.
\end{lem}
\begin{proof}
Denote by $(\Phi,\theta)\in\X_\delta\times X_\delta$, the functions associated to $T$ through Lemma \ref{solutionU}.
Further note that
$$\partial_T \Psi_0^T|_{T=1}=C_p(T-T_0)^{-1} \gf_0(\rho),$$
for some constant $C_p>0$.
Therefore we can rewrite $\Uf(T,\vf)$ as
\begin{align*}
\Uf(T,\vf)=&J(\vf)^T+ (T-1)C_p (T-T_0)^{-1} \gf_0+(T-1)^2(T-T_0)^{2}\ff(T)\\
=&J(\vf)^T+ (T-1)C_p (T-T_0)^{-1}\gf_{\theta_\infty}+(T-1)C_p (T-T_0)^{-1} (\gf_{\theta(0)}-\gf_{\theta_\infty})\\
&+(T-1)^2(T-T_0)^{2}\ff(T)
\end{align*}
with $\|\ff\|\lesssim 1.$
As $|\theta(0)-\theta_\infty|\lesssim \delta$, we have that 
$$
\|\gf_{\theta(0)}-\gf_{\theta_\infty}\|\lesssim \delta
$$
and from this we infer that 
\begin{align*}
\left(\Pf_{\theta_{\infty}} \Uf(T,\vf)|\gf_{\theta_\infty}\right)= (T-1)C_p (T-T_0)^{-1} \|\gf_{\theta_\infty}\|^2+O(\frac{\delta}{c})+O(c^{-2}),
\end{align*}
where each of the $O$-terms is a continuous function of $T$.
Hence Lemma \ref{solutionU} implies that 
\begin{align*}
\left(\Cf(\Phi,\theta,\Uf(T,\vf))|\gf_{\theta_\infty}\right)= (T-1)C_p (T-T_0)^{-1} \|\gf_{\theta_\infty}\|^2+O(\frac{\delta}{c})+O(c^{-2}).
\end{align*}
Consequently the vanishing of $\Cf(\Phi,\theta,\Uf(T,\vf))$ is equivalent to $T$ solving the equation 
\begin{align*}
1-T=(T-T_0)\left(O(\frac{\delta}{c})+O(c^{-2})\right).
\end{align*}
Note that the right hand side, denoted by $F(T) $, is continuous in $T$ and satisfies
$$
|F|\leq \frac{\delta}{c^2}.
$$
It follows that $1+F$ is a continuous map from $[1-\frac{\delta}{c^2},1+\frac{\delta}{c^2}]$ to $[1-\frac{\delta}{c^2},1+\frac{\delta}{c^2}]$
and such a map necessarily has a fixed point.
\end{proof}
Now we are able to prove our main result.
\subsection{Proof of Theorem \ref{maintheorem}}
\begin{proof}
Theorem \ref{maintheorem} is now essentially a consequence of the last few Lemmas. Therefore let $\delta,T_0,$ and $c$, be as in the assumptions of Lemma \ref{Cvanish}
and suppose that the initial data $(f,g)$ of Eq.~\eqref{startingeq} satisfy the conditions of Theorem \ref{maintheorem}.
Then, by the previous Lemma, there exists a $T \in [1-\frac{\delta}{c^2},1+\frac{\delta}{c^2}]$ as well as functions $(\Phi,\theta) \in \mathcal{X}_{\delta} \times X_\delta$ which solve Eq.~\eqref{duhameleq} with initial data $\Phi(0)= \Uf(T,\vf)$,
where $\vf:=(\tilde{f},\tilde{g})$ is as in \ref{splitting}.
Hence, $\Psi(\tau):=\Phi(\tau)+\Psi_{\theta(\tau)}$ satisfies Eq.~\eqref{abstraceq} in the mild sense, with initial data $\Psi(0)= \Psi_0+ U(T,\vf)$.
By undoing the transformations done in the first section, we obtain that
\begin{align*}
&u(t,r)=(T-t)^{-\frac{2}{p-1}}(\psi_1+i\psi_3)(-\log(T-t)+\log (T-T_0),\frac{r}{T-t})
\end{align*}
solves the original perturbed radial wave equation (\ref{startingeq}) with initial data
\begin{align*}
u(T_0,r)=&u^1_0(T_0,r)+\tilde{f}(r)\\
\partial_0 u(T_0,r)=&\partial_0 u^1_0(T_0,r)+\tilde{g}(r).
\end{align*}
We thus calculate
\begin{align*}
(T-t)^{\frac{2}{p-1}+\frac{1}{2}}&\|u_{\theta_\infty}(t)-u(t,.)\|_{H^2(\B^3_{T-t})}
\\
\lesssim& (T-t)^{\frac{1}{2}}\|\Psi_{\theta_\infty}-\Psi(-\log(T-t)+\log (T-T_0),\frac{.}{T-t}))\|_{\mathcal{H}(\B^3_{T-t})}\\
\leq&\|\Psi_{\theta_\infty}-\Psi(-\log(T-t)+\log (T-T_0),.)\|\\
\leq&\|\Psi_{\theta_\infty}-\Psi_{\theta(-\log(T-t)+\log (T-T_0))}\|\\
&+|(\Phi(-\log(T-t)+\log (T-T_0))\| \\
\lesssim& (T-t)^{\omega_p},
\end{align*}
for all $t\in [T_0,T).$
As the other stated bounds follow analogously, the proof of Theorem \ref{maintheorem} is completed.
\end{proof}

\bibliography{references}
\bibliographystyle{plain}

\end{document}